\documentclass[11pt]{article}
\usepackage{amsmath,amssymb,amsthm,color,graphicx}
\usepackage[latin1]{inputenc}
\usepackage{amsfonts}
\usepackage{wasysym}
\newcommand{\R}{\mathbb R}

\newcommand{\vr}{\vec{r}}
\newcommand{\vp}{\vec{p}}
\newcommand{\vq}{\vec{q}}

\newtheorem{theorem}{Theorem}[section]
\newtheorem{lemma}{Lemma}[section]
\newtheorem{proposition}{Proposition}[section]
\newtheorem{cor}{Corollary}[section]
\newtheorem{remark}{Remark}[section]
\newtheorem{defi}{Definition}[section]
\newtheorem{acknowledgment*}{Acknowledgment}
\newtheorem{assumption}{Assumption}[section]
\newcommand{\be}{\begin{equation}}
\newcommand{\ee}{\end{equation}}

\begin{document}






\huge
\begin{center}{\bf Semi-Discrete approximation of  Optimal  Mass Transport}\end{center}
\normalsize
\begin{center}{ G. Wolansky, \  \\ Department of Mathematics, Technion, Haifa 32000, Israel \footnote{Email: \  gershonw@math.tecnion.ac.il}}\end{center}
\begin{abstract}
 Optimal mass transport is described by an approximation of transport cost  via  semi-discrete  costs. The notions of optimal partition and optimal strong partition are given  as well.  We also suggest an algorithm for computation of Optimal Transport for general cost functions induced by an action, an asymptotic error estimate and several numerical examples of optimal partitions.
\end{abstract}
\section{Introduction}
   Optimal mass transport (OMT) goes back to the pioneering paper of Monge [\ref{mon}] at the 18th century.
  In 1942, L. Kantorovich [\ref{[Kan]}] observed that   OMT  can be relaxed into an infinite dimensional linear programming in measure spaces. As such ,  it  has a dual formulation which is very powerful and  was later (1987) used by Brenier [\ref{[Br]}] to develop the theory of Polar factorization of positive  measures.
    OMT has many connections with PDE, kinetic theory, fluid dynamics, geometric inequalities, probability and many other fields in mathematics as well as in computer science and economy.

Even though finite dimensional (or discrete) OMT  is well understood, its extension to infinite dimensional measure spaces  poses a great challenge, e.g. uniqueness and  regularity theory of fully non-linear PDE such as the Monge-Amper equation [\ref{Cap1}].
  \par
  We suggest to investigate a bridge between finite  ("discrete") and infinite ("continuum") dimensional OMT. This notion of {\it semi-discrete} OMT  leads naturally to {\it optimal partition} of measure spaces. Our motivation in this paper  is the development of numerical method for solving OMT.   Efficient algorithms  are of great interest  to  many fields in operational research and, recently, also for optical design [\ref{GO}, \ref{Ru1}, \ref{Ru2}] and computer vision ("earth moving metric") [\ref{[RYT]}].

When dealing with numerical approximations for OMT, the problem must be reduced to a discrete, finite   OMT (with, perhaps, very large number of degrees of freedom). Discrete OMT   is often called the {\it assignment problem}.  This is, in fact,  a general title for a variety of linear and quadratic programming. It seems that the first efficient algorithm was the so called "Hungarian Algorithm", after two Hungarian mathematicians. See [\ref{[HWK]}, \ref{[VO]}, \ref{[HWK1]}, \ref{[AF]}, \ref{[JM]}]
and the survey paper  [\ref{[PD]}] for many other relevant references.
\par
The deterministic, finite assignment problem is easy to formulate.  We are given $n$ men and $n$ women. The cost of matching  man $i$ to a woman $j$ is  $c_{i,j}$. The object is to find the assignment (matching) $i\rightarrow j$, given in terms of a permutation $j=\tau(i)$ which minimize the total cost of matching $\sum_{i=1}^nc_{i, \tau(i)}$.
\par
When replacing the deterministic assignment by a probabilistic one, we assign the probability   $p_i^j\geq 0$ for matching man $i$ to woman $j$.  The discrete assignment problem is then  reduced to the linear programming of minimizing \be\label{dismonge}\sum_{i=1}^n\sum_{j=1}^n p_i^j c_{i,j}\ee over all stochastic $n\times n$ matrices $P:=\{p_i^j\}$, i.e. these matrices which  satisfy the $2n+n^2$ linear constraints $$\sum_{k=1}^n p_k^j=\sum_{k=1}^np_i^k=1 \ \ ; \ \ p_i^j\geq 0 \  \ \forall \ i,j\in \{1, \ldots n\} \ .  $$  The  Birkhoff Theorem assures us, to our advantage,  that the optimal solution of this continuous assignment problem is also the solution of the deterministic version.
\par
The probabilistic version seems to be more difficult since it involves a search on a much larger set of $n\times n$ stochastic matrices.  On the other hand, it  has a clear advantage since  it is, in fact, a linear programming which can be handled effectively by well developed algorithms for such problems.
\par
In many cases the probabilistic version cannot be reduced to the deterministic problem. For example, if the number of sources   $n$ and number of targets  $m$ not necessarily  equal, or when  not all sources must find  target, and/or not all targets  must be met, then the constraints are relaxed into $\sum_{i=1}^n p_i^j\leq 1$ and/or $\sum_{i=1}^mp_j^i\leq 1$. We shall not deal with these extension in the current paper, except, to some extent, in section \ref{hedonic} below.
\subsection{ From the discrete assignment problem  to the continuum OMT}
Let $\mu$ be a probability measure on some measure space $X$, and $\nu$ another probability measure on (possibly different) measure space $Y$. Let $c=c(x,y)$ be the cost of transporting $x$ to $y$. The object of the Monge problem is to find a measurable mapping $T:X\rightarrow Y$ which generalizes the deterministic assignment perturbation $\tau$ described above in the following sense:
\be\label{push} T_\#\mu=\nu \ \ \text{namely} \ \ \mu(T^{-1}(B))=\nu(B)\ee
for every  $\nu-$measurable set $B\subset Y$.  The optimal Monge mapping (if exists) realizes the infimum  $$ \inf_{T_\#\mu=\nu}\int_X c(x,T(x)) \mu(dx) \ . $$
The relaxation of Monge problem into Kantorovich problem is analogues to the relaxation of the deterministic assignment problem to  the probabilistic one: Find the minimizer  \be\label{infmonge} c(\mu,\nu):= \min_{\pi\in\Pi_X^Y(\mu,\nu)}\int_X\int_Y c(x,y)\pi(dxdy)\ee
 among all probability measures $\pi\in \Pi_X^Y(\mu,\nu):=$   \be\label{Pi} \{ \text{ Probability measures on}\ X\times Y \  \text{whose}  \ X \ (\text{resp.}\  Y) \ \text{marginals are} \ \mu\ (\text{resp.}\ \nu)\} \ . \ee
\par
In fact, Kantorovich problem is just an infinite dimensional linear programming over the huge set $\Pi_X^Y(\mu,\nu)$. The Monge problem can be viewed as a restriction of the Kantorovich problem to the class of {\it deterministic} probability measures in $\Pi_X^Y(\mu,\nu)$, given by $\pi(dxdy)=\mu(dx)\delta_{y-T(x)}$ where $T_\#\mu=\nu$.  It turns out, somewhat surprisingly, that the value $c(\mu,\nu)$ of the  Kantorovich problem equals to the infimum (\ref{infmonge}) of Monge problem, provided  $c$ is a continuous function on $X\times Y$ and $\mu$ does not contain a Dirac $\delta$ singularity (an atom) [\ref{[A1]}].

\subsection{Semi-finite approximation- The middle way}
Suppose the transportation cost $c=c(x,y)$ on $X\times Y$ can be obtained by interpolation  of pair of functions $c^{(1)}$ on $X\times Z$ and $c^{(2)}$ on $Z\times Y$, where $Z$ is a third domain and the interpolation means
\be\label{interpol} c(x,y):=\inf_{z\in Z}c^{(1)}(x,z)+c^{(2)}(z,y) \ . \ee
   A canonical example   for  $X=Y=\R^d$ is $c(x,y)=c(|x-y|)$ where  $c(w)=|w|^p$,  $p\geq 1$. Then  (\ref{interpol}) is valid for $Z=\R^d$ and both $c^{(1,2)}(w)=2^{p-1}|w|^p$. So
   \be\label{interpole} c(x,y):=|x-y|^p=2^{p-1}\inf_{z\in\R^d} |x-z|^p+|z-y|^p\ee
   for any $x,y\in\R^d$ provided $p\geq 1$. Note in particular that the minimizer above is unique, $z=(x+y)/2$, provided $p>1$, while $z=tx+(1-t)y$ for any $t\in[0,1]$ if $p=1$.
\par
Let $Z= Z_m:=\{z_1, \ldots z_m\}\subset Z$ is a finite set. Denote
  \be\label{sd} c^{Z_m}(x,y):= \min_{z\in Z_m}c^{(1)}(x,z)+c^{(2)}(z,y)\geq c(x,y)\ee
   {\it the ($Z_m$) semi-finite approximation} of $c$ given by (\ref{interpol}). \par
An optimal transport plan for a semi-discrete cost (\ref{sd}) is obtained as a pair of $m-${\it partitions} of the spaces $X$ and $Y$.
An $m-$partition is a decomposition of the the space into $m$ mesurable, mutually disjoint subset. It turns out that $c^{Z_m}(\mu,\nu)$ can be obtained as
\be\label{optpar} c^{Z_m}(\mu,\nu)=\inf_{\{A_z\}, \{B_z\}}\sum_{z\in Z_m}\int_{A_z}c^{(1)}(x,z)\mu(dx) + \int_{B_z}c^{(2)}(z,y)\nu(dz) \ee
where the infimum is on the pair of partitions $\{A_z\}$ of $X$ and $\{B_z\}$ of $Y$ satisfying $\mu(A_z)=\nu(B_z)$ for any $z\in Z_m$.   The optimal plan is, then, reduced to $m$ plans transporting $\bar{A}_z\subset X$ to $\bar{B}_z\subset Y$, for any $z\in Z_m$, where $\{\bar{A}_z,\bar{B}_z\}$ is the optimal partition realizing (\ref{optpar}).
\par
The real advantage of the semi-discrete method described above is that it has a dual formulation which convert the optimization (\ref{optpar}) to a convex optimization on $\R^m$. Indeed, we prove that for a given $Z_m\subset Z$  there exists a concave function $\Xi_{\mu, Z_m}^\nu:\R^m\rightarrow \R$ such that
$$ \max_{\vp\in\R^m}\Xi_{\mu, Z_m}^\nu(\vp)=c^{Z_m}(\mu,\nu) $$
and, under some conditions on either $\mu$ or $\nu$,  the maximizer is unique up to a uniform translation $\vp \rightarrow \vp+\beta(1, \ldots 1)$ on $\R^m$. Moreover,  the maximizers of $\Xi_{\mu, Z_m}^\nu$ yield the {\it unique} partitions $\{A_z, B_z; \ z\in Z_m\}$ of (\ref{optpar}).
\par

 The accuracy of the approximation of $c(x,y)$ by $c^{Z_m}(x,y)$ depends, of course, on the choice of the set $Z_m$.  In the special (but interesting) case $X=Y=Z=\R^d$ and $c(x,y)=|x-y|^\sigma$, $\sigma>1$ it can be shown that   $c^{Z_m}(x,y)-c(x,y)=O(m^{-2/d})$ for any $x,y$ in a compact set,  where $Z_m$ are  distributed on a regular grid containing this set. 
 \par
 From (\ref{sd}) and the above reasoning we obtain in particular
 \be\label{deltac} c^{Z_m}(\mu,\nu)-c(\mu,\nu)\geq 0 \  \ee
 for any pair of probability measures, and that, for a reasonable choice of $Z_m$,
  (\ref{deltac}) is of order $m^{-2/d}$ if the supports of $\mu,\nu$ are contained in a compact set.

 For a given $m\in\mathbb{N}$ and  pair of probability measures $\mu,\nu$ and ,  the optimal choice of $Z_m$ is the one which minimizes (\ref{deltac}). Let
 \be\label{cm} \phi^m(\mu,\nu):= \inf_{Z_m\subset Z}c^{Z_m}(\mu,\nu)-c(\mu,\nu) \geq 0 \  \ee
 where the infimum is over all sets of $m$ points in $Z$. Note that the optimal choice now depends on the measures $\mu,\nu$ themselves (and not only on their supports).  A natural question is then to evaluate  the assymptotic limits
 $$ \bar{\phi}(\mu,\nu):= \limsup_{m\rightarrow \infty} m^{2/d}\underline{c}^m(\mu,\nu) \ \ \ ; \ \ \  \underline{\phi}(\mu,\nu):= \liminf_{m\rightarrow \infty} m^{2/d}\underline{c}^m(\mu,\nu) \ . $$
 Some preliminary results regarding these limits are discussed in this paper.
\subsection{Numerical method}
The numerical calculation of (\ref{infmonge}) we advertise in this paper apply the semi-discrete approximation  $c^{Z_m}$  of order $m$. It also  involves  discretization  of $\mu,\nu$ into atomic measures of finite support ($n$). The level of approximation is determined by the two parameters: The cardinality of the supports of the discretized measures, $n$, and the cardinality of the semi-finite approximation $m$ of the cost. The idea of semi-discrete approximation is to  choose  $n$ much larger than $m$. As we shall see, the evaluation of the approximate solution  involves finding a maximizer to a concave function in $m$ variables, where the complexity of calculating this function, and each of its partial derivatives,  is of order $n$. A naive gradient descent method then result in $O(m)$ iterations to approximate this maximum, where each iteration is of order $mn$. This yields a complexity of order $O(m^2n)$ to obtain a transport plan on the approximation level of $m^{-2/d}$. This should be  compared to the $n^3$ complexity  of the Hungarian algorithm [\ref{PCK}]. We shall not, however, pursue a rigorous complexity estimate in this paper.

\subsection{Structure of the paper}
In section \ref{secOptpar} we consider optimal partitions in the weak sense of probability measures, as Kantorovich relaxation of  solutions of the optimal transport in semi-discrete setting. We formulate and prove a duality theorem (Theorem \ref{firstdu}) which yields the relation between the minimizer of the OMT with semi-discrete cost to maximizing  a dual function $\Xi$  of $m$ variables.
\par
In section \ref{secstrong} we define strong partitions of the domains,  and introduce conditions for the uniqueness of optimal solution and its representation as the analogue of optimal Monge mapping. The main results of this section is given in Theorem \ref{strongth}. In section \ref{hedonic} we introduce an interesting application of this concept to the theory of pricing of goods in Hedonic markets, and remark on possible generalization of optimal partitions to {\it optimal subpartition}. This model, related generalizations and further analysis will be pursued in a separate publication.
\par
In section \ref{secmonotone} we discuss optimal sampling of fixed number of centers ($m$). In particular we show a monotone sequence of improving semi-discrete approximation by floating the $m$ centers into improved positions. In section \ref{secass} we provide   some assymptotic properties of the  error of the semi-discrete approximation as $m\rightarrow\infty$.
\par
In section \ref{desal}
 we introduces a detailed description of the algorithm on the discrete level.
\par
In section \ref{yifat} we show some numerical experiments of calculating optimal partitions in the case of quadratic cost functions on a planar domain.
\par
The numerical method we propose in this paper has some common features with the approach of Merigot [\ref{Me}], see also [\ref{Levi}], as we recently discovered. We shall discuss this issues in section \ref{compare}.

\subsection{Notations and standing assumptions}
\begin{enumerate}
\item $X$, $Y$ are Polish (complete, separable)  metric spaces.
\item  ${\cal M}_+(X)$ is  the cone of non-negative Borel measures on $X$ (resp. for $Y$).
\item The $weak-*$ topology on ${\cal M}_+(X)$ is the dual of $C_b(X)$, the space of bounded continuous functions on $X$ (resp. for $Y$).
\item  ${\cal M}_1(X)$  is the cone of probability (normalized) non-negative Borel measures in ${\cal M}_+(X)$ (resp. for $Y$).
    \item For $\mu\in{\cal M}_1(X)$, $\nu\in{\cal M}_1(Y)$, \\ $\Pi_X^Y(\mu,\nu):= \{\pi\in {\cal M}_1(X\times Y) \ ; \mu \ \text{is the}
     \ X \ \text{marginal and} \ \nu \ \text{is the}
     \ Y \ \text{marginal of } \ \pi\}$
    \item The $m-$simplex $\Sigma_m:=\{\vec{s}:= (s_1, \ldots s_m), \ s_i\geq 0, \ \sum_{i=1}^m s_i=1 \}\subset \R^m$.

\end{enumerate}

\section{Optimal partitions}\label{secOptpar}

 \begin{defi}\label{defP}\noindent
 \begin{description}
 \item{i)}
 A $m-$partition of a pair of  a probability measure $\mu\in {\cal M}_1(X)$ subjected to $\vec{r}\in\Sigma_m$  is given by $m$ nonnegative measures $\mu_z\in {\cal M}_+(X)$ on $X$  such that $\sum_{z\in Z_m}\mu_z=\mu$ and $\int_Xd\mu_z=r_z$. The set of all such partitions $\vec{\mu}:= (\mu_1, \ldots \mu_m)$ is denoted by ${\cal P}_X^{\vec{r}}(\mu)$.
 \item {ii)} If, in addition, $\nu\in{\cal M}_1(Y)$ then $(\vec{\mu},\vec{\nu})\in {\cal P}_X^Y(\mu,\nu)$ iff $\vec{\mu}\in {\cal P}_X^{\vec{r}}(\mu)$ and $\vec{\nu}\in {\cal P}_Y^{\vec{r}}(\nu)$ for {\em some}  $\vec{r}\in\Sigma_m$.
 \end{description}
 \end{defi}
 The following Lemma is a result of compactness of probability Borel measure on a compact space (see e.g. [\ref{[B]}]).
 \begin{lemma}\label{comp}
 For any $\vec{r}\in\Sigma_m$, the set of partitions ${\cal P}_X^{\vec{r}}$ is compact with respect to the $(C^*)^m(X)$ topology. In addition, ${\cal P}_X^Y(\mu,\nu)$ is compact with respect to $(C^*)^m(X)\times (C^*)^m(Y)$ topology.
 \end{lemma}
  \begin{lemma}\label{lemapart}
  $$ c^{Z_m}(\mu,\nu)=\min_{(\vec{\mu}, \vec{\nu})\in{\cal P}_X^Y(\mu,\nu)}\sum_{z\in Z_m} \left[\int_X c^{(1)}(x,z)\mu_z(dx) + \int_Y c^{(2)}(z,y)\nu_z(dy\right] \  $$
  where $c^{Z_m}(\mu,\nu)$ as defined by (\ref{infmonge}, \ref{sd}) and $(\vec{\mu}, \vec{\nu})\in {\cal P}_X^Y(\mu,\nu)$.
  \end{lemma}
  \begin{proof}
  First note that the existence of minimizer  is obtained by Lemma \ref{comp}. \par
  Define, for $z\in Z_m$,  $$\Gamma_z:= \{(x,y)\in X\times Y; \ c^{(1)}(x,z)+c^{(2)}(z,y)\leq c^{Z_m}(x,y) \}\subset X\times Y$$
  such that $\Gamma_z$ is measurable in $X\times Y$, $\Gamma_z\cap\Gamma_{z^{'}}=\emptyset$ if $z\not= z^{'}$ and $\sum_{z\in Z_m}\Gamma_z=X\times Y$. Note that, in general, the choice of $\{\Gamma_z\}$ is not unique.
  \par
  Given $\pi\in\Pi_X^Y(\mu,\nu)$, let $\pi_z$ be the restriction of $\pi$ to $\Gamma_z$. In particular $\sum_{z\in Z_m}\pi_z=\pi$. Let $\mu_z$ be the $X$ marginal of $\pi_z$ and $\nu_z$ the $y$ marginal of $\pi_z$.  Then $(\vec{\mu},\vec{\nu})$ defined in this way is in ${\cal P}_X^Y(\mu,\nu)$. Since by definition $c^{Z_m}(x,y)=c^{(1)}(x,z)+ c^{(2)}(z,y)$ a.s. $\pi_z$,  \begin{multline}\int_X\int_Yc^{Z_m}(x,y) \pi(dxdy) = \sum_{z\in Z_m} \int_X\int_Y c^{Z_m}(x,y)\pi_z(dxdy)\\
  =\sum_{z\in Z_m} \int_X\int_Y (c^{(1)}(x,z)\pi_z(dxdy)+ \int_X (c^{(2)}(z,y)\pi_z(dxdy)\pi_z(dxdy) \\
  =\sum_{z\in Z_m}\left[ \int_X c^{(1)}(x,z)\mu_z(dx)+ \int_Y c^{(2)}(z,y)\nu_z(dy)\right]\end{multline}
  Choosing $\pi$ above to be the optimal transport plan we get the inequality
  $$  c^{Z_m}(\mu,\nu)\geq\inf_{(\vec{\mu}, \vec{\nu})\in{\cal P}_X^Y(\mu,\nu)}\sum_{z\in Z_m} \left[\int_X c^{(1)}(x,z)\mu_z(dx) + \int_Y c^{(2)}(z,y)\nu_z(dy\right] \ . $$
  To obtain the opposite inequality, let $(\vec{\mu},\vec{\nu})\in {\cal P}_X^Y(\mu,\nu)$ and set $r_z:=\int_Xd\mu_z\equiv\int_Y d\nu_z$. Define
  $\pi(dxdy)=\sum_{z\in Z_m}r_z^{-1}\mu_z(dx)\nu_z(dy)$.  Then $\pi\in\Pi_X^Y(\mu,\nu)$ and, from (\ref{sd})
  \begin{multline} \int_X\int_Y c^{Z_m}(x,y)\pi(dxdy)= \sum_{z\in Z_m}\int_X\int_Y c^{Z_m}(x,y)r_z^{-1}\mu_z(dx)\nu_z(dy)
  \\ \leq \sum_{z\in Z_m}\int_X (c^{(1)}(x,z)+ c^{(2)}(z,y))r_z^{-1}\mu_z(dx)\nu_z(dy)\\
  = \sum_{z\in Z_m}\left[ \int_X c^{(1)}(x,z)\mu_z(dx)+ \int_Y c^{(2)}(z,y)\nu_z(dy)\right]
  \end{multline}
  and we get the second inequality.
  \end{proof}
  Given $\vp=(p_{z_1}, \ldots p_{z_m})\in\R^m$, let
  \be\label{xi}  \xi_{Z_m}^{(1)}(\vp,x):= \min_{z\in Z_m} c^{(1)}(x,z)+p_z \ \ ; \ \ \  \xi^{(2)}_{Z_m}(\vp,y):= \min_{z\in Z_m} c^{(2)}(z,y)+p_z\ee
  \be\label{Xi12} \Xi_\mu^{Z_m}(\vp):=\int_X \xi_{Z_m}^{(1)}(\vp,x)\mu(dx) \ \ ; \ \ \Xi_\nu^{Z_m}(\vp):=\int_Y \xi^{(2)}_{Z_m}(\vp,y)\nu(dy) \ . \ee
  \be\label{Xi} \Xi^\nu_{\mu, Z_m}(\vp):= \Xi_\mu^{Z_m}(\vp) + \Xi_\nu^{Z_m}(-\vp) \ . \ee

 \begin{lemma} \label{lemdu}If $\mu\in {\cal M}_1(X)$  then for any  $\vec{r}\in\Sigma_m$,
 \be\label{Xi*mu}(-\Xi_\mu^{Z_m})^*(-\vec{r}):= \sup_{\vp\in\R^m}\Xi_\mu^{Z_m}(\vp)-\vp\cdot\vec{r}= c^{(1)}\left(\mu, \sum_{z\in Z_m} r_z\delta_{z}\right) = \min_{\vec{\mu}\in {\cal P}_X^{\vec{r}}(\mu)}\sum_{z\in Z_m}\int_X c^{(1)}(x,z)\mu_z(dx)\ . \ee
 Analogously, for  $\nu\in{\cal M}_1(Y)$
  \be\label{Xi*nu}(-\Xi_\nu^{Z_m})^*(-\vec{r}):= \sup_{\vp\in\R^m}\Xi_\nu^{Z_m}(\vp)-\vp\cdot\vec{r}= c^{(2)}\left(\nu, \sum_{z\in Z_m}\delta_{z}\right)= \min_{\vec{\nu}\in {\cal P}_Y^{\vec{r}}(\nu)}\sum_{z\in Z_m}\int_Y c^{(2)}(z,y)\nu_z(dy)\ \ . \ee
 Here $\vp\cdot\vec{r}:= \sum_{z\in Z_m} r_zp_z$.
 \end{lemma}
 \begin{proof}
 This is a special case of the general duality theorem of Monge-Kantorovich. See, for example [\ref{[V1]}]. It is also a special case of generalized partitions,
 see Theorem 3.1 and its proof in [\ref{[W]}].

 \end{proof}
 \begin{theorem}\label{firstdu}
 \be\label{ap} \sup_{\vp\in\R^m} \Xi_{\mu, Z_m}^\nu(\vp)= c^{Z_m}(\mu,\nu) \ . \ee
 \end{theorem}
 \begin{proof}
 From Lemma \ref{lemapart}, Lemma \ref{lemdu} and Definition \ref{defP} we obtain
 \be\label{c=Xi} c^{Z_m}(\mu,\nu)= \inf_{\vec{r}\in\Sigma_m}\left[ (-\Xi_\mu^{Z_m})^*(-\vec{r})+(-\Xi_\nu^{Z_m})^*(-\vec{r})\right] \ . \ee
 Note that $(-\Xi_\mu^{Z_m})^*$, $(-\Xi_\nu^{Z_m})^*$ as  defined in ( \ref{Xi*mu}, \ref{Xi*nu}), are, in fact, the Legendre transforms of $-\Xi_\mu^{Z_m}$, $-\Xi_\nu^{Z_m}$, respectively. As such, they are defined formally on the whole domain $\R^m$ (considered as the dual of itself under the canonical  inner product). It follows that $(-\Xi_\mu^{Z_m})^*(\vec{r})=(-\Xi_\nu^{Z_m})^*(\vec{r})=\infty $   for $\vec{r}\in\R^m-\Sigma_m$. Note that this definition is consistent with the right hand side of ( \ref{Xi*mu}, \ref{Xi*nu}), since ${\cal P}_X^{\vec{r}}(\mu)={\cal P}_Y^{\vec{r}}(\nu)=\emptyset$ for $\vec{r}\not\in\Sigma_m$.

 On the other hand, $\Xi_\mu^{Z_m}$ and $\Xi_\nu^{Z_m}$ are both finite and continuous on the whole of $\R^m$. The Fenchel-Rockafellar duality theorem (see [\ref{[V1]}]- Thm 1.9) then implies
 \be\label{Xi*=Xi} \sup_{\vp\in\R^m} \Xi_\mu^{Z_m}(\vp)+ \Xi_\nu^{Z_m}(-\vp) = \inf_{\vec{r}\in \R^m} (-\Xi_\mu^{Z_m})^*(\vec{r})+ (-\Xi_\nu^{Z_m})^*(\vec{r})\ . \ee
 The proof follows from (\ref{Xi}, \ref{c=Xi}).
 \par\noindent
 {\bf An alternative proof}: \\ We can prove (\ref{ap}) directly by constrained minimization, as follows: $(\vec{\mu}, \vec{\nu})\in {\cal P}_X^Y(\mu,\nu)$ iff $F(\vp, \phi,\psi):=$
 $$\sum_{z\in Z_m} p_i\left(\int_X d\mu_i-\int_Y d\nu_i\right) + \int_X \phi(x)\left(\mu(dx)-\sum_{z\in Z_m} \mu_z(dx)\right) + \int_Y \psi(y)\left(\nu(dy)-\sum_{z\in Z_m} \nu_z(dy)\right)\leq 0$$
 for any choice of $\vp\in\R^m$, $\phi\in C(X)$, $\psi\in C(Y)$. Moreover, $\sup_{\vp,\phi,\psi}F=\infty$ unless $(\vec{\mu}, \vec{\nu})\in {\cal P}_X^Y(\mu,\nu)$.
  We can then obtain from Lemma \ref{lemapart}: $c^{Z_m}(\mu,\nu)=$
   \begin{multline}\label{multi}\inf_{\{\mu_z\in {\cal M}_+(X), \nu_z\in{\cal M}_+(Y)\}}\sup_{\vp\in\R^m, \phi\in C(X), \psi\in C(Y)}\sum_{z\in Z_m} \left[\int_X c^{(1)}(x,z)\mu_z(dx) + \int_Y c^{(2)}(z,y)\nu_z(dy)\right]   + F(\vp, \phi, \psi) \ \\
   =  \sup_{\vp\in\R^m, \phi\in C(X), \psi\in C(Y)}\inf_{\{\mu_z\in {\cal M}_+(X), \nu_z\in{\cal M}_+(Y)\}}
    \sum_{z\in Z_m} \int_X\left( c^{(1)}(x,z)+p_z-\phi(x)\right)\mu_z(dx)
  \\ + \sum_{z\in Z_m} \int_Y\left( c^{(2)}(z,y)-p_z-\psi(y)\right)\nu_z(dy) +\int_X\phi \mu(dx) +\int_Y\psi\nu(dy) \ . \end{multline}
 We now observe that the infimum on $\{\mu_z, \nu_z\}$
 above is $-\infty$ unless $c^{(1)}(x,z)+p_z-\phi(x)\geq 0$ and $c^{(2)}(z,y)+p_z-\psi(y)\geq 0$ for any $z\in Z_m$. Hence, the two sums on the right of  (\ref{multi}) are non-negative, so the infimum   with respect to $\{\mu_z, \nu_z\}$ is zero. To obtain the supremum on the last two integrals on the right of (\ref{multi}) we choose $\phi, \psi$ as large  as possible under this constraint, namely
 $$ \phi(x)=\min_{z\in Z_m} c^{(1)}(x,z)+p_z \ \ \ , \ \ \ \psi(y)=\min_{z\in Z_m} c^{(2)}(z,y)-p_z$$
 so $\phi(x)\equiv \xi_{Z_m}^{(1)}(\vp,x)$, $\psi(y)\equiv\xi_{Z_m}^{(2)}( -\vp,y)$  by definition via (\ref{xi}).
 \end{proof}
 \section{Strong partitions}\label{secstrong}
 We now define strong partitions as a special case of partitions (Definition \ref{defP}).
  \begin{defi}\label{sdefP}\noindent
 \begin{description}
 \item{i)} A partition $\vec{\mu}\in {\cal P}_X^{\vec{r}}(\mu)$ is called a {\em strong $m-$partition} if there exists $m$ measurable sets $A_z\subset X$, $z\in Z_m$ which are essentially disjoint, namely $\mu(A_z\cap A_{z^{'}})=\emptyset$ for $z\not= z^{'}$ and $\mu(\cup_{z\in Z_m} A_z)=X$, such that $\mu_z$ is the restriction of $\mu$ to $A_z$. The set of strong $m-$partition corresponding to $\vec{r}\in\Sigma_m$ is denoted by $\widehat{\cal P}_X^{\vec{r}}(\mu)$.
 \item {ii)} In addition, for  $\nu\in{\cal M}_1(Y)$ then $(\vec{\mu},\vec{\nu})\in \widehat{\cal P}_X^Y(\mu,\nu)$ iff $\vec{\mu}\in \widehat{\cal P}_X^{\vec{r}}(\mu)$ and $\vec{\nu}\in \widehat{\cal P}_Y^{\vec{r}}(\nu)$ for {\em some}  $\vec{r}\in\Sigma_m$. In particular, a strong $m-$partition is composed of $m$ $\mu$ measurable sets $A_z\subset X$ and $m$  $\nu$ measurable sets $B_z\subset Y$ such that
     $\int_{A_z} d\mu=\int_{B_z}d\nu$ for $z\in Z_m$.
 \end{description}
 \end{defi}
 \begin{assumption}\label{assAB}
 \begin{description} .
 \item{a)} $\mu\in{\cal M}_1(X)$ is atomless and
 $\mu(x; c^{(1)}(x,z)-c^{(1)}(x, z^{'})=p)=0$ for any $p\in\R$ and any $z, z^{'}\in Z_m$.
 \item{b)} $\nu\in{\cal M}_1(Y)$ is atomless and $\nu(y; c^{(2)}(z,y)-c^{(2)}(z^{'},y)=p)=0$ for any $p\in\R$ and any $z, z^{'}\in Z_m$.
 \end{description}
 \end{assumption}
 Let us also define, for $\vp\in\R^m$
 \be\label{sparAB}A_z(\vp):= \{ x\in X; \ c^{(1)}(x,z)+p_z= \xi^{(1)}_{Z_m}(\vp,x)\} \ \ \ ; \ \ \ B_z(\vp):= \{ y\in Y; \ c^{(2)}(z,y)+p_z= \xi^{(2)}_{Z_m}(\vp,y)\} \ . \ee
 Note that, by (\ref{xi}, \ref{Xi12})
 \be\label{Ximu0} \Xi_\mu^{Z_m}(\vp)= \sum_{z\in Z_m}\int_{A_z(\vp)} (c^{(1)}(x,z)+p_z)\mu(dx)\ee
 likewise
  \be\label{Xinu0} \Xi_\nu^{Z_m}(\vp)= \sum_{z\in Z_m}\int_{B_z(\vp)} (c^{(2)}(z,y)+p_z)\nu(dy) \ . \ee
 \begin{lemma}\label{Xidiff}Under assumption \ref{assAB} (a) (resp. (b))
 \begin{description}
 \item{i)} For any $\vp\in\R^m$, $\{A_z(\vp)\}$ (resp. $\{B_z(\vp)\}$) induces  essentially disjoint partitions of $X$ (resp. $Y$).
     \item{ii)} $\Xi_\mu^{Z_m}$  (resp. $\Xi_\nu^{Z_m}$) is continually  differentiable functions on $\R^m$,
     $$ \frac{\partial \Xi_\mu^{Z_m}}{\partial p_z}=\mu(A_z(\vp)) \ \ \ \text{resp.}  \ \ \ \frac{\partial \Xi_\nu^{Z_m}}{\partial p_z}=\nu(B_z(\vp)) \ . $$
 \end{description}
 \end{lemma}
This Lemma is a special case of  Lemma 4.3 in  [W].
 \begin{theorem}\label{strongth}
  Under either  assumption \ref{assAB}-(a) or (b) there exists a unique
 minimizer $\vec{r}_0$ of (\ref{c=Xi}). In addition, there exists a maximizer $\vp_0\in\R^m$ of $\Xi_{\mu, Z_m}^\nu$,  and either (in case (a)) $\{A_z(\vp_0)\}$ or  (in case (b)) $\{B_z(-\vp_0)\}$  induces a corresponding strong $m-$partition in (a) $\widehat{\cal P}_X^{\vec{r}_0}(\mu)$ or
 (b) $\widehat{\cal P}_Y^{\vec{r}_0}(\nu)$. In particular, if {\em both} (a+b) holds then  $\{A_z(\vp_0), \{B_z(-\vp_0)\}$ induces a strong $m-$partition in
 $\widehat{\cal P}_X^Y(\mu,\nu)$, and
 \be\label{pirep} \pi_0(dxdy):= \sum_{z\in Z_m; r_{0,z}=\mu(A_z(\vp_0))} (r_{0,z})^{-1}{\bf 1}_{A_z(\vp_0)}(x){\bf 1}_{B_z(-\vp_0)}(y) \mu(dx)\nu(dy) \ee
 is the unique optimal transport plan for  $c^{Z_m}(\mu, \nu)$.
\end{theorem}
\begin{proof}
Note that $\Xi(\vp)-\vec{r}\cdot\vp$ is invariant under additive shift for $\Xi=\Xi_\mu^{Z_m}, \Xi_\nu^{Z_m}$ and $\vec{r}\in \Sigma_m$. Indeed,
$\Xi(\vp+\alpha \vec{1})=\Xi(\vp)+\alpha$ for any $\alpha\in\R$ where $\vec{1}:=(1, \ldots 1)$. So, we restrict the domain of $\Xi$ to
\be\label{normal} \vp\in R^m \ \ , \ \ \vp\cdot\vec{1}=0 \ .  \ee
Assume (a). Given $\vec{r}\in\Sigma_m$.
Assume first
\be\label{posr}r_z\in (0,1) \ \ \ \text{for any} \  z\in Z_m \ . \ee  We  prove the existence of a maximizer $\vp_0$,
$$(-\Xi_\mu^{Z_m})^*(-\vec{r})=\Xi_\mu^{Z_m}(\vp_0)-\vp_0\cdot\vec{r}\geq \Xi_\mu^{Z_m}(\vp)-\vp\cdot\vec{r}$$
for any $\vp\in\R^m$. Let $\vp_n$ be a maximizing sequence, that is
$$\lim_{n\rightarrow\infty} \Xi_\mu^{Z_m}(\vp_n)-\vp_n\cdot\vec{r}=(-\Xi_\mu^{Z_m})^*(-\vec{r})$$
(c.f. (\ref{Xi*nu})).
 \par
 Let $\|\vp\|_2:= (\sum_{z\in Z_m}p^2_z)^{1/2}$ be the Euclidian norm of $\vp=(p_{z_1}, \ldots p_{z_m})\in\R^m$. If we prove that for any maximizing sequence $\vp_n$ the norms $\|\vp_n\|_2$ are uniformly bounded, then there exists  a converging subsequence whose limit is the maximizer $\vp_0$. This follows, in particular, since $\Xi_\mu^{Z_m}$ is a closed (upper-semi-continuous) function.
 \par
 Assume there exists a subsequence along which $\|\vp_n\|_2\rightarrow\infty$. Let $\hat{\vp}_n:= \vp_n/\|\vp_n\|_2$.
Let
 \begin{multline}\label{try1}\Xi_\mu^{Z_m}(\vp_n)-\vp_n\cdot\vec{r}:= \left[\Xi_\mu^{Z_m}(\vp_n)-\vp_n\cdot\nabla_{\vp}\Xi_\mu^{Z_m}(\vp_n)\right] + \vp_n\cdot\left( \nabla_{\vp}\Xi_\mu^{Z_m}(\vp_n)-\vec{r}\right)\\
 =\left[\Xi_\mu^{Z_m}(\vp_n)-\vp_n\cdot\nabla_{\vp}\Xi_\mu^{Z_m}(\vp_n)\right] + \|\vp_n\|_2\hat{\vp}_n\cdot\left( \nabla_{\vp}\Xi_\mu^{Z_m}(\vp_n)-\vec{r}\right) \ . \end{multline}
In addition, by (\ref{Ximu0}) and Lemma\ref{Xidiff}-(ii)
 \begin{multline}\label{try2}-\infty < \int_X \min_{z\in Z_m} c^{(1)}(x,z) \mu(dx) \leq \left[\Xi_\mu^{Z_m}(\vp)-\vp\cdot\nabla_{\vp}\Xi_\mu^{Z_m}(\vp)\right]=\sum_{z\in Z_m} \int_{A_z(\vp)}c^{(1)}(x,z)\mu(dx) \\ \leq \int_X \max_{z\in Z_m} c^{(1)}(x,z) \mu(dx)<\infty \ . \end{multline}
 By (\ref{try1}- \ref{try2}) we obtain, for $\|\vp_n\|_2\rightarrow\infty$,
 \be\label{lim0phat}\lim_{n\rightarrow\infty} \hat{\vp}_n\cdot\left( \nabla_{\vp}\Xi_\mu^{Z_m}(\vp_n)-\vec{r}\right)=0 \ . \ee
   Since $\hat{\vp}_n$ lives in the unit sphere $S^{m-1}$ in $\R^m$ (which is a compact set), there exists a subsequence for which $\hat{\vp}_n\rightarrow \hat{\vp}_0:= (\hat{p}_{0,z_1}, \ldots \hat{p}_{0,z_m})\in S^{m-1}$. Let $P_-:= \min_{z\in Z_m} \hat{p}_{z,0}$ and $J_-:= \{ z\in Z_m \ ; \hat{p}_{0,z}=P_-\}$.

 Note that for $n\rightarrow\infty$ along such a subsequence, $p_{n,z}-p_{n, z\prime}\rightarrow-\infty$ for $z\in J_-, z\prime\not\in J_-$. It follows that $A_{z\prime}(\vp_n)=\emptyset$ if $z\prime\not\in J_-$ for $n$ large enough, hence $\cup_{z\in J_-}A_z(\vp_n) =X$ for $n$ large enough. Let
 $\mu_z^n$ be the restriction of $\mu$ to $A_z(\vp_n)$. Then the limit $\mu^n_z\rightharpoonup \mu_z$ exists (along a subsequence) where $n\rightarrow\infty$.  In particular, by Lemma \ref{Xidiff}
 $$\lim_{n\rightarrow\infty} \frac{\partial\Xi_\mu^{Z_m}}{\partial p_{n,z}}(\vp_n)= \int_X d\mu_z$$
 while  $\mu_z\not=0$ if only if $z\in J_-$, and $\sum_{z\in J_-}\mu_z=\mu$. Since $\hat{\vp}_{0,z}=P_-$ for $z\in J_-$ is the minimal value of the coordinates of $\hat{\vp}_0$, it follows that
 $$\lim_{n\rightarrow\infty} \hat{\vp}_n\cdot\left( \nabla_{\vp}\Xi_\mu^{Z_m}(\vp_n)-\vec{r}\right)=-\vec{r}\cdot\hat{\vp}_0 +P_-\sum_{z\in J_-}\int_Xd\mu_z=-\vec{r}\cdot\hat{\vp}_0 +P_- \ . $$
Now, by (\ref{posr}), $\vec{r}\cdot \vp_0>P_-$ unless $J_-=Z_m$. In the last case we obtain a contradiction of (\ref{normal}) since it implies $\hat{\vp}_0=0$ which contradicts $\hat{\vp}_0\in S^{m-1}$. If $J_-$ is a proper subset of $Z_m$ we obtain a contradiction to (\ref{lim0phat}).
\par
If (\ref{posr}) is violated we may restrict to domain of $\Xi_\mu^{Z_m}$ to a subspace by eliminating all coordinates $z\in Z_m$ for which $r_z=0$. On the restricted subspace we have a minimizer $\vp_0$ by the above proof. Then we may extend $\vp_0$ by  assigning  $p_z$ sufficiently small if $r_z=0$. This guarantees  $A_z(\vp_0)=\emptyset$, hence (Lemma \ref{Xidiff}) $\partial\Xi_\mu^{Z_m}/\partial p_z=0$ for any such $z$. Hence the extended $\vp_0$ is still a critical point of $\Xi_\mu^{Z_m}(\vp)-\vec{r}\cdot\vp$, and is a maximizer by concavity of $\Xi_\mu^{Z_m}$.
\par
Next, we prove that $A_z(\vp_0)$ is a unique  optimal partition of $X$.  Let $\vec{\mu}\in {\cal P}_X^r$ be a minimizer of  (\ref{Xi*mu}).  Since $\int_Xd\mu_z=r_z$, $\sum_{z\in Z_m}\mu_z=\mu$,  (\ref{Xi*mu}) implies
$$  \sum_{z\in Z_m}\int_X c^{(1)}(x,z)\mu_z(dx) = (-\Xi_\mu^{Z_m})^*(-\vr)$$
and
 $$(-\Xi_\mu^{Z_m})^*(-\vr)=\Xi_\mu^{Z_m}(\vp_0)-\vr\cdot\vp_0 = \int_X  \xi^{(1)}_{Z_m}( \vp_0,x)d\mu - \vp_0\cdot\vr=\sum_{z\in Z_m}\int_X \left( \xi^{(1)}_{Z_m}( \vp_0,x) - p_{0,z}\right)d\mu_z \ ,$$
so
$$ \sum_{z\in Z_m}\int_X\left( \xi^{(1)}_{Z_m}(\vp_0,x)-p_{0,z} -c^{(1)}(x,z)\right)\mu_z(dx)=0 \ . $$
On the other hand, $\xi^{(1)}_{Z_m}(\vp_0,x)-p_{0,z} -c^{(1)}(x,z)\leq 0$ for any $x\in X$ by definition (\ref{xi}), so we must have the equality
$$ \xi^{(1)}_{Z_m}(\vp_0,x)=p_{0,z} +c^{(1)}(x,z)$$
a.e. on supp$(\mu_z)$. Hence supp$(\mu_z)\subset A_z(\vp_0)$.  Since $A_z(\vp_0)$ are mutually disjoint and $\sum_{z\in Z_m}\mu_z=\mu$, then $\mu_z$ is necessarily the restriction of $\mu$ to $A_z(\vp_0)$.  On the other hand, for any $\vp\not=\vp_0\mod \mathbb{R}{\vec{1}}$ there exists $z\in Z_m$ for which $\mu\left(A_z(\vp_0)\Delta A_z(\tilde{\vp})\right)\not=0$. This implies that the strong partition $\vec{A}(\vp_0)$ is the unique one.

The same result  is applied to $\Xi_\nu^{Z_m}(\vp)-\vp_0\cdot\vr$. If we show that the minimizer $\vr_0$ of the right side of  (\ref{Xi*=Xi}) is unique, then it follows that the maximizer $\vp_0$ of the left side of (\ref{Xi*=Xi}) is unique as well (up to $\vec{1}\mathbb{R}$), and, in particular, the optimal partition is unique. Hence, we only have to show the uniqueness of the minimizer of the right side of (\ref{Xi*=Xi}). This, in turn, follows if  either $(-\Xi_\mu^{Z_m})^*$ or $(-\Xi_\nu^{Z_m})^*$ is {\it strictly convex}.
\par
 To prove this we recall some basic elements form convexity theory (see, e.g. [BC]):
  \begin{description}
  \item{i)} \ If $F$ is a convex function on $\R^m$ (say), then  the  sub gradient $\partial F$ at point $p\in\R^m$ is defined as follows: $\vq\in \partial F(\vp)$ if and only if
 $$ F(\vp\prime)-F(\vp)\geq \vq\cdot(\vp\prime-\vp) \ \ \ \forall \vp\prime\in\R^m \ . $$
\item{ii)} The Legendre transform of $F$:
$$ F^*(\vq):= \sup_{\vp\in\R^m} \vp\cdot\vq-F(\vp) \ ,$$
and $Dom(F^*)\subset\R^m$ is the set on which $F^*<\infty$.
  \item{iii)} The function $F^*$ is convex (and closed), but $Dom(F^*)$ can be a proper subset of $\R^m$ (or even an empty set).
  \item {iv)} The subgradient of a convex function is non-empty (and convex) at any point in the proper domain of this function (i.e. at any point in which the function takes a value in $\R$).
\item{v)} \
Young's  inequality
$$ F(\vp)+F^*(\vq)\geq \vp\cdot\vq$$
holds for any pair of points $(\vp, \vq)\in \R^m\times\R^m$. The equality holds iff $\vq\in\partial F(\vp)$, iff $\vp\in\partial F^*(\vq)$.
\item{vi)} \
The Legendre transform is involuting, i.e $F^{**}=F$ if $F$ is convex and closed.
 \item{vii)} A convex function is continuously differentiable in the interior of its proper domain iff its subgradient at any point in the interior of its domain is a singleton.
\end{description}
Returning to our case, let $F:=-\Xi_\mu^{Z_m}$. It is a closed, convex, proper and continuously differentiable  function defined everywhere on  $\R^m$. Assume $(-\Xi_\mu^{Z_m})^*$ is not strictly convex. It means there exists $\vr_1\not= \vr_2\in Dom(-\Xi_\mu^{Z_m})^*$ for which \be\label{just}(-\Xi_\mu^{Z_m})^*(\frac{\vr_1+\vr_2}{2})=\frac{(-\Xi_\mu^{Z_m})^*(\vr_1)+(-\Xi_\mu^{Z_m})^*(\vr_2)}{2} \ . \ee
Let $\vr:=\vr_1/2+\vr_2/2$, and $\vp\in \partial(-\Xi_\mu^{Z_m})^*(\vr)$. Then, by (iv, v)
\be\label{just1} 0=(-\Xi_\mu^{Z_m})^*(\vr)+ (-\Xi_\mu^{Z_m})^{**}(\vp)-\vp\cdot\vr=(-\Xi_\mu^{Z_m})^*(\vr) -\Xi_\mu^{Z_m}(\vp)-\vp\cdot\vr \ . \ee
By (\ref{just}, \ref{just1}):
$$\frac{1}{2}\left( (-\Xi_\mu^{Z_m})^*(\vr_1)-\Xi_\mu^{Z_m}(\vp)-\vp\cdot\vr_1 \right)+ \frac{1}{2}\left( (-\Xi_\mu^{Z_m})^*(\vr_2)-\Xi_\mu^{Z_m}(\vp)-\vp\cdot\vr_2 \right)=0$$
while (v) also guarantees
$$(-\Xi_\mu^{Z_m})^*(\vr_i)-\Xi_\mu^{Z_m}(\vp)-\vp\cdot\vr_i\geq 0 \ \ \ , \ \ i=1,2 \ . $$
It follows
$$(-\Xi_\mu^{Z_m})^*(\vr_i)-\Xi_\mu^{Z_m}(\vp)-\vp\cdot\vr_i= 0 \ \ \ , \ \ i=1,2 \ , $$
so, by (v) again,  $\{\vr_1, \vr_2\}\in \partial\Xi_\mu^{Z_m}(\vp)$. This is a contradiction of (vii) since $\Xi_\mu^{Z_m}$ is continuously differentiable everywhere on $\R^m$ by Lemma \ref{Xidiff}.
\par
Finally, we prove that $\pi_0$ given by (\ref{pirep}) is an optimal plan.
First observe that $\pi_0\in\Pi(\mu,\nu)$, hence
$$ c^{Z_m}(\mu,\nu)\leq \int_X\int_Y c^{Z_m}(x,y)\pi_0(dxdy) \ . $$
 Then we get, from (\ref{sd})
$$c^{Z_m}(\mu,\nu)\leq  \int_X\int_Y c^{Z_m}(x,y)\pi_0(dxdy)\leq \sum_{z\in Z_m}\int_{A_z(\vp_0)\times B_z(-\vp_0)} (c^{(1)}(x,z)\mu(dx) + c_{(2)}(y,z)\nu(dy))$$
$$ = \sum_{z\in Z_m}\left( \int_{A_z(\vp_0)} c^{(1)}(x,z)\mu(dx) + \int_{B_z(-\vp_0)} c^{(2)}(z,y)\nu(dy)\right) = \Xi(\vp_0)\leq c^{Z_m}(\mu,\nu) \ $$
where the last equality from Theorem \ref{firstdu}. In particular, the first inequality is an equality so $\pi_0$ is an optimal plan indeed.
\end{proof}
\section{Pricing in hedonic market}\label{hedonic}
In adaptation to the  model of Hedonic market [\ref{[Hed]}] there are 3 components: The space of  consumers  (say, $X$), space of producers (say $Y$) and space of commodities, which we take here to be a finite set $Z_m:=\{z_1, \ldots z_m\}$.    The function $c^{(1)}:= c^{(1)}(x,z)$ is  the {\it negative} of the utility
of commodity $z\in Z_m$ to consumer $x$, while  $c^{(2)}:=c^{(2)}(z,y)$ is the cost of producing commodity $z\in Z_m$ by the producer  $y$.
\par
Let $\mu$ be a probability measure on $X$ representing the distribution of  consumers, and $\nu$ a probability measure on $Y$ representing the distribution   of the producers. Following [\ref{[Hed]}] we add the "null commodity" $z_0$ and assign the zero utility and cost
 $c^{(1)}(x, z_0)=c^{(2)}( z_0,y)\equiv 0$ on $X$ (resp. $Y$). We understand the meaning that a consumer (producer) chooses the null commodity is that he/she avoids consuming (producing) any item from $Z_m$.

The object of pricing in Hedonic market is to find   equilibrium prices for the commodities which will balance supply and demand: Given  a price $p_z$ for $z$, the consumer at $x$ will buy the commodity $z$ which minimize   its loss  $c^{(1)}(x,z)+p_z$, or will  buy nothing (i.e. "buy" the null commodity $z_0$) if $\min_{z\in Z_m}c^{(1)}(x,z)+p_z>0$),  while producer at $y$ will prefer to produce commodity $z$ which maximize its profit $-c^{(2)}(z,y) + p_z$, or will produce nothing if $\max_{z\in Z_m}-c^{(2)}(z,y)+p_z<0$. Using notation (\ref{xi}-\ref{Xi}) we define
 \be\label{xi0}  \xi_X^0(\vp,x):= \min\{\xi^{(1)}_{Z_m}(\vp,x), 0\} \ \ ; \ \ \  \xi^0_Y(\vp,y):= \min\{\xi^{(2)}_{Z_m}(\vp,y),0\}\ee
  \be\label{Xi120} \Xi^0_\mu(\vp):=\int_X \xi_X^0(\vp,x)\mu(dx) \ \ ; \ \ \Xi^0_\nu(\vp):=\int_Y \xi_Y^0(\vp,y)\nu(dy) \ . \ee
  \be\label{Xi0} \Xi^{0,\nu}_\mu(\vp):= \Xi^0_\mu(\vp) + \Xi^0_\nu(-\vp) \ . \ee
  Thus, $\Xi^{0,\nu}_\mu(\vp)$ is the difference between the total loss of all consumers and  the total profit of all producers, given the prices vector $\vp$. It follows that an equilibrium price vector balancing supply and demand is the one which (somewhat counter-intuitively) {\it maximizes} this difference.
The corresponding  optimal strong $m-$partition represent the matching between producers of  ($B_z\subset Y$) to consumers ($A_z\subset X$) of $z\in Z$.  The introduction of null commodity allows the possibility that only part of the consumer (producers) communities actually consume (produce), that is $\cup_{z\in Z_m} A_z\subset X$ and $\cup_{z\in Z_m} B_z\subset Y$, with $A_0=X-\cup_{z\in Z_m} A_z$ ($B_0=Y-\cup_{z\in Z_m}B_z$) being the set of non-buyers (non-producers).
\par
From the dual point of view, an adaptation  $c^{Z_m}_0(x,y):=\min\{ c^{Z_m}(x,y), 0\}$ of (\ref{sd}) (in the presence of null commodity) is the {\it cost of direct  matching} between producer $y$ and  consumer $x$. The {\it optimal matching} $(A_z, B_z)$ is the one which {\it minimizes} the total cost  $c_0^{Z_m}(\mu,\nu)$  over all  {\it sub-}$m-$partitions $\widehat{\cal P}_X^Y(\mu,\nu)$ as defined in Definition \ref{sdefP}-(ii) with the possible inequality $\mu(\cup A_z)=\nu(\cup B_z)\leq 1$.
\section{Dependence on the sampling set}\label{secmonotone}
So far we took the smapling set $Z_m\subset Z$ to be fixed. Here we consider the effect of optimizing $Z_m$ within the sets of cardinality $m$ in $Z$.
\par
 As we already know from (\ref{interpol}, \ref{sd}), $c^{Z_m}(x,y)\geq c(x,y)$ on $X\times Y$  for any $(x,y)\in X\times Y$ and $Z_m\subset Z$. Hence also $c^{Z_m}(\mu,\nu)\geq c(\mu,\nu)$ for any $\mu,\nu\in {\cal M}_1$ and any $Z_m\subset Z$ as well. An {\it improvement} of $Z_m$ is a new choice $Z_m^{new}\subset Z$ of the {\it same} cardinality $m$ such that $c^{Z_m^{new}}(\mu,\nu)<c^{Z_m}(\mu,\nu)$.
 \par
 In  section \ref{monim} we propose a way to improve a given $Z_m\subset Z$, once the optimal partition is calculated. Of course, the improvement depends on the measure $\mu,\nu$.
 \par
 In section \ref{secass} we discuss the limit $m\rightarrow \infty$ and prove some assymptotic estimates.
\subsection{Monotone improvement}\label{monim}
\begin{proposition}\label{promon}
 Define $\Xi_{\mu, Z_m}^\nu$ on $\R^m$ with respect to $Z_m:= \{z_1, \ldots z_m\}\in Z$ as in (\ref{Xi}). Let $(\vec{\mu}, \vec{\nu})\in {\cal P}_X^Y(\mu,\nu)$ be the optimal partition
 corresponding to $c^{Z_m}(\mu,\nu)$. Let $\zeta(i)\in Z$ be a minimizer of
 \be\label{Znew} Z\ni \zeta \mapsto \int_X c^{(1)}(x,\zeta)\mu_{z_i}(dx) + \int_Y c^{(2)}(\zeta,y)\nu_{z_i}(dy)  \ .  \ee
   Let $Z_m^{new}:= \{ \zeta(1), \ldots \zeta(m)\}$.  Then $c^{Z_m^{new}}(\mu,\nu)\leq c^{Z_m}(\mu,\nu)$.
\end{proposition}
\begin{cor}\label{cormon}
Let Assumption \ref{assAB} (a+b), and $\vp_0$ be the minimizer of $\Xi_\mu^{\nu, Z_m}$ in $\R^m$. Let  $\{A_z(\vp_0), B_z(-\vp_0)\}$  be the strong partition corresponding to $Z_m$ as in (\ref{sparAB}). Then the components of  $Z_m^{new}$ are obtained as the minimizers of
 $$ Z\ni\zeta \mapsto \int_{A_z(\vp_0)} c^{(1)}(x,\zeta)\mu(dx) + \int_{B_z(-\vp_0)} c^{(2)}(\zeta,y)\nu(dy) \ . $$
\end{cor}
\begin{proof}(of Proposition \ref{promon}): Let $\Xi_\mu^{\nu, new}$  be defined with respect to $Z_m^{new}$. By Lemma \ref{lemapart} and Theorem \ref{firstdu}
 $\Xi^{\nu, new}_\mu(\vp)\leq \Xi^\nu_\mu(\vp^*):= \max_{\R^m}\Xi_\mu^{\nu, Z_m}$ for any $\vp\in\R^m$,  \\ so  $\max_{\R^m}\Xi_\mu^{\nu,new}(\vp)\equiv c^{Z_m^{new}}(\mu,\nu) \leq \max_{\R^m}\Xi_\mu^{\nu, Z_m}(\vp)\equiv c^{Z_m}(\mu,\nu)$.
\end{proof}
\begin{remark}\label{remmean} If $c$ is a quadratic cost then $z^{new}$ is the center of mass of $A_z(\vp_0)$ and $B_z(-\vp_0)$:
 $$ z^{new}:= \frac{\int_{A_z(\vp_0)}x\mu(dx) + \int_{B_z(-\vp_0)}y\nu(dy)}{\mu(A_z(\vp_0)) + \nu(B_z(-\vp_0))} \ . $$
We shall take advantage of this in section \ref{apqc}. \end{remark}
Let
$$ \underline{c}^m(\mu,\nu):= \inf_{Z_m\subset Z\ ; \ \#(Z_m)=m} c^{Z_m}(\mu,\nu) \ . $$
Let $Z_m^k:=\{ z_1^k, \ldots z_m^k\}\subset Z$ be a sequence of sets such that $z_z^{k+1}$ is obtained from $Z_m^k$ via  (\ref{Znew}). Then by Proposition \ref{promon}
$$ c(\mu,\nu)\leq \underline{c}^m(\mu,\nu) \leq \ldots c^{Z_m^{k+1}}(\mu,\nu)\leq c^{Z_m^{k}}(\mu,\nu) \leq \ldots c^{Z_m^{0}}(\mu,\nu)  \ . $$
{\bf Open problem:} Under which additional conditions one may gurantee
$$ \lim_{k\rightarrow\infty}c^{Z_m^{k}}(\mu,\nu)= \underline{c}^m(\mu,\nu)  \ \ ? $$

\subsection{Assymptotic estimates}\label{secass}
Recall the definition (\ref{cm})
 $$ \phi^m(\mu,\nu):= \inf_{Z_m\subset Z}c^{Z_m}(\mu,\nu)-c(\mu,\nu) \geq 0 \ \ .  $$
 Consider the case  $X=Y=Z=\R^d$ and
 $$c(x,y)=\min_{z\in\R^d}h(|x-z|)+ h(|y-z|)$$
 where $h:\R^+\rightarrow \R^+$ is convex, monotone increasing, twice continuous  differentiable.  Note that $c(x,y) = 2h (|x-y|/2)$.
 \begin{lemma}\label{ssy}
 Suppose both $\mu$ and $\nu$ are supported on in a compact set in $\R^d$. Then there exists $C=C(\mu,\nu)<\infty$  such that
 \be\label{ssyeq} \limsup_{m\rightarrow \infty} m^{2/d}\phi^m(\mu,\nu)\leq C(\mu,\nu) \ . \ee
 \end{lemma}
 \begin{proof}
 By Taylor expansion of $z\rightarrow h(|x-z|)+ h(|y-z|)$ at $z_0=(x+y)/2$ we get
 $$ h(|x-z|)+ h(|y-z|)= 2h(|x-y|/2) + \frac{1}{2|x-y|^2}h^{''}\left(\frac{|x-y|}{2}\right)\left[ (x-y)\cdot(z-z_0)\right]^2+o^2(z-z_0)\ \ . $$
 Let now $Z_m$ be a regular grid of $m$ points which contains the support $K$. The distance between any $z\in K$ to the nearest point in the grid does not exceed $C(K)m^{-1/d}$, for some constant $C(K)$. Hence $c_m(x,y)-c(x,y)\leq \sup |h^{''}| C(K)^2m^{-2/d}$ if $x,y\in K$. Let $\pi_0(dxdy)$ be the optimal plan corresponding to $\mu,\nu$ and $c$. Then, by definition,
 $$ c(\mu,\nu)=\int_X\int_Y c(x,y)\pi_0(dxdy)  \ \ ; \ \  c_m(\mu,\nu)\leq \int_X\int_Y c_m(x,y)\pi_0(dxdy)$$
 so $$\phi^m(\mu,\nu)\leq  \int_X\int_Y (c_m(x,y)-c(x,y))\pi_0(dxdy)\leq \sup |h^{''}| C(K)^2m^{-2/d} \ , $$
 since $\pi_0$ is a probability measure.
 \end{proof}
 If $h(s)=2^{\sigma-1}s^\sigma$ (hence $c(x,y)=|x-y|^\sigma$) then the condition of Lemma \ref{ssy} holds if $\sigma\geq 2$. Note that if $\mu=\nu$ then $c(\mu,\mu)=0$ so $\phi^m(\mu,\mu)=\inf_{Z_m\in Z} c^{Z_m}(\mu,\mu)$.  In that particular case we can improve the result of Lemma \ref{ssy} as follows:

\begin{proposition}\label{prozador}
If $c(x,y)=|x-y|^\sigma$, $\sigma \geq 1$, $X=Y=Z=\R^d$ and  $\nu=\mu=f(x)dx$
\be\label{mquant}  \lim_{m\rightarrow \infty}m^{\sigma/d}\phi^m(\mu,\mu)= C_{d,\sigma}\left(\int f^{d/(d+\sigma)}dx\right)^{(d+\sigma)/d} \ \ee
where $C_{d,\sigma}$ is some universal constant.
\end{proposition}
\begin{proof}
From (\ref{Xi}), $\Xi_{\mu, Z_m}^\mu(\vp)= \Xi_\mu^{Z_m}(\vp)+\Xi_\mu^{Z_m}(-\vp)$ is an even function. Hence its maximizer must be $\vp=0$. By Theorem \ref{firstdu}
$$\Xi_\mu^{\mu, Z_m}(0)= c^{Z_m}(\mu,\mu) \ . $$
Using (\ref{xi}, \ref{Xi12}) with $c^{(1)}(x,y)=c^{(2)}(y,x)= 2^{\sigma-1}|x-y|^\sigma$ we get
$$ \Xi_\mu^{\mu, Z_m}(0)= 2^\sigma\int_{\R^d}\min_{z\in z_m}|x-z|^\sigma \mu(dx) \ . $$
We then obtain (\ref{mquant}) from Zador's Theorem [\ref{[GL]}, \ref{[Z]}, [\ref{[GL]}].
\end{proof}
Note that Proposition \ref{prozador} does not contradict Lemma \ref{ssy}. In fact $\sigma\geq 2$ it is compatible with the Lemma,  and (\ref{ssyeq}) holds with $C(\mu,\mu)=0$
if $\sigma>2$. If $\sigma\in [1,2)$, however, then the condition of the Lemma is not satisfied (as $h^{''}$ is not bounded near $0$), and the Proposition is a genuine extension of the Lemma, in the particular case $\mu=\nu$.
\par
We can obtain a somewhat sharper result for {\it any} pair $\mu,\nu$ in the case $\sigma=2$, which is presented below.
\par
Let $X=Y=Z=\R^d$,  $c(x,y)=|x-y|^2$, $\mu, \nu$ are Borel probability measures which admits a finite second moment. Assume $\mu$ is asbolutely continuous with respect to Lebesgue measure on $\R^d$. In that case, Brenier Polar factorization Theorem [\ref{[Br]}] implies the existence of a unique solution to the quadratic Monge problem, i.e a Borel mapping  $T$ such that $T_\#\mu=\nu$.  Let $\lambda=f(x)dx$ be the McCann  interpolation between $\mu$ and $\nu$, that is, $\lambda=(I/2+T/2)_\#\mu$. We  know that $\lambda$ is absolutely continuous with respect to Lebesgue as well.
\begin{theorem}
Under the above assumptions,
$$  \limsup_{m\rightarrow \infty}m^{2/d}\phi^m(\mu,\nu)\leq 4C_{d,2}\left(\int f^{d/(d+2)}dx\right)^{(d+2)/d} \ . $$
\end{theorem}
\begin{proof}
Let $S$ be the optimal Monge mapping transporting $\lambda$ to $\nu$, i.e. $S_\#\lambda=\nu$ is a solution of Monge problem
$$ \int_{\R^d} |S(x)-x|^2\lambda(dx)=\min_{Q; Q_\#\lambda=\nu}\int_{\R^d}|Q(x)-x|^2\lambda(dx)\ . $$
Note that if $y=(T(x)+x)/2$ then $S(y)=T(x)$. Then, since $\lambda=(I/2+T/2)_\#\mu$,
$$ \int_{\R^d} \left|S(y)-y\right|^2 \lambda(dy)=\int_{\R^d}\left| \frac{T(x)-x}{2}\right|^2 \mu(dx) \equiv c(\mu,\nu)/4\ . $$
Also,  if $y=(T(x)+x)/2$ then  $2y-S(y)=x$. If follows that
 $2I-S$ is the optimal Monge mapping  transporting $\lambda$ to $\mu$, that is,
$$ \int_{\R^d} |S(x)-2x|^2\lambda(dx)=\int_{\R^d} |S(x)-x|^2\lambda(dx)=\min_{Q; Q_\#\lambda=\mu}\int_{\R^d}|Q(x)-x|^2\lambda(dx)\  $$
so
\be\label{sumcc}c(\mu,\nu)=2\int_{R^d}|S(x)-x|^2\lambda(dx)+  2\int_{R^d}|S(x)-2x|^2\lambda(dx)= 4\int_{R^d}|S(x)-x|^2\lambda(dx) \ .  \ee

Given $z\in Z_m$, let
\be\label{Vzdef}V_z:= \left\{ x\in \R^d; \ \ |x-z|\leq |x-z^{'}| \ \ \forall z^{'}\in Z_m \right\} \ . \ee
Since $\cup_{z\in Z_m} V_z=\R^d$ and $\lambda(V_z\cap V_{z^{'}})=0$ for $z\not= z^{'}$ then (\ref{sumcc}) implies
\be\label{ss1}c(\mu,\nu)=4\sum_{z\in Z_m}\int_{V_z}|S(x)-x|^2\lambda(dx) \ . \ee
Let $\nu_z:= S_\#\lambda\lfloor V_z$, $\mu_z:= (2I-S)_\#\lambda\lfloor V_z$.
Form Lemma  \ref{lemapart}
$$ c^{Z_m}(\mu,\nu)\leq 2\left(\sum_{z\in Z_m} \int |x-z|^2\mu_z(dx)+ \sum_{z\in Z_m} \int |x-z|^2\nu_z(dx)\right)$$
\be\label{ss2}= 2\sum_{z\in Z_m}\int_{V_z}\left\{|S(x)-z|^2+ |2x-S(x)-z|^2\right\}\lambda(dx) \ee
By the identity
$$ 4 |z-x|^2 =  2\left\{|S(x)-z|^2+ |2x-S(x)-z|^2\right\} - 4|S(x)-x|^2  \ . $$
This, together with (\ref{ss1}, \ref{ss2}) and (\ref{cm}) implies
\be\label{rree} \phi^m(\mu,\nu)\leq 4\sum_{z\in Z_m} \int_{V_z}|x-z|^2 \lambda(dx) \ .  \ee
By (\ref{Vzdef}),
 $\sum_{z\in Z_m}\int_{V_z}|x-z|^2 \lambda(dx)=\int_{\R^d}\min_{z\in Z_m}|x-z|^2\lambda(dx):= \phi(\lambda, Z_m)$. Since (\ref{rree}) is valid for any $Z_m$ we get
the result  from  Zador's Theorem [\ref{[GL]}, \ref{[Z]}, [\ref{[GL]}].
\end{proof}

 \section{Description of the Algorithm}\label{desal}
 We now spell out the proposed algorithm for approximating  of the optimal plan $c(\mu,\nu)$. We assume that $c$ is given by (\ref{interpol}). We fix a large numbers $n_1, n_2$ (not necessarily equal)  which characterizes the fine sampling, and much smaller $m$ characterizing the partition order. Then we choose an appropriate sampling: In $X$ we set  $\mu_{n_1}:= \sum_{i=1}^{n_1} s_i\delta_{x_i}$  for $\mu$ and on $Y$ we set
 $\nu_{n_2}:= \sum_{i=1}^{n_2} \tau_i\delta_{y_i}$ for $\nu$.
 \par
 At the first  stage we  choose $Z^{(0)}:= \{z^0_1, \ldots z^0_m\}\in Z^m$, and define
 $$ \Xi_0(\vp):= \sum_{i=1}^{n_1} s_i\min_{1\leq j\leq m}[ c^{(1)}(x_i, z^0_j) + p_j]+  \sum_{i=1}^{n_2} \tau_i\min_{1\leq j\leq m}[ c^{(2)}( z^0_j,y_i) - p_j]
 $$
Next we choose a favorite method to maximize $\Xi_0$ on $\R^m$. It is helpful to observe that $\Xi_0$ is differentiable a.e. on $\R^m$. Indeed,
let
 $$ A^0_j(\vp):= \{ i\in (1, \ldots n_1), \ c^{(1)}(x_z, z^0_j) + p_j = \min_{1\leq k\leq m}[ c^{(1)}(x_z, z^0_k) + p_k]\}$$
  $$ B^0_j(\vp):= \{ i\in (1, \ldots n_2), \ c^{(2)}( z^0_j, y_z) + p_j = \min_{1\leq k\leq m}[ c^{(2)}( z^0_k, y_z) + p_k]\} \ . $$
Then
$$ \frac{\partial \Xi_0}{\partial p_j} = \sum_{i\in A_j(\vp)} s_z-\sum_{i\in B_j(-\vp)} \tau_z$$
provided $A_j(\vp)\cap A_k(\vp)=\emptyset$ and $B_j(-\vp)\cap B_k(-\vp)=\emptyset$ for any $k\not= j$.
\par
Let $\vp_0$ be the maximizer of $\Xi_0$ on $\R^m$, $A^0_j:= A^0_j(\vp_0)$, $B^0_j:= B^0_j(-\vp_0)$.
\par
At the $l$ step we are given $Z^{(l)}:= \{z^l_1, \ldots z^l_m\}\in Z^m$,  $\vp_l$ the maximizer of
 $$ \Xi_l(\vp):= \sum_{i=1}^{n_1} s_z\min_{1\leq j\leq m}[ c^{(1)}(x_z, z^l_j) + p_j]+  \sum_{i=1}^{n_2} \tau_z\min_{1\leq j\leq m}[ c^{(2)}( z^l_j, y_z) - p_j]
 $$
 and the corresponding $A^l_j:= A^l_j(\vp_l)$, $B^l_j:= B^l_j(-\vp_l)$ where
 $$ A^l_j(\vp):= \{ i\in (1, \ldots n_1), \ c^{(1)}(x_z, z^l_j) + p_j = \min_{1\leq k\leq m}[ c^{(1)}(x_z, z^l_k) + p_k]\}$$
  $$ B^l_j(\vp):= \{ i\in (1, \ldots n_2), \ c^{(2)}(z^l_j, y_z) + p_j = \min_{1\leq k\leq m}[ c^{(2)}( z^l_k, y_z) + p_k]\} \ . $$
 We define $z^{l+1}_j$ as the minimizer of
 \be\label{newc} \zeta \mapsto \sum_{i\in A^l_j} s_z c^{(1)}(x_z, z) + \sum_{i\in B^l_j} \tau_z c^{(2)}( z, y_z)  \ee
 and set  $Z^{(l+1)}:= \{z^{l+1}_1, \ldots z^{l+1}_m\}\in Z^m$. Now
 $$ \Xi_{l+1}(\vp):= \sum_{i=1}^{n_1} s_z\min_{1\leq j\leq m}[ c^{(1)}(x_z, z^{l+1}_j) + p_j]+  \sum_{i=1}^{n_2} \tau_z\min_{1\leq j\leq m}[ c^{(2)}( z^{l+1}_j, y_z) - p_j] \ .
 $$
 From these we evaluate the maximizer $\vp_{l+1}$ the maximizer of $\Xi_{l+1}$ and the sets $A_j^{l+1}$, $B_j^{l+1}$.
 \begin{remark}
 The maximizer $\vp_l$ at the $l$ stage can be used as an initial guess for calculating the maximizer $\vp_{l+1}$ at the next stage. This can save a lot of iterations where the stages where changes of the centers $Z^{(l)}\rightarrow Z^{(l+1)}$ is small.
 \end{remark}
 Using Proposition \ref{promon} and Corollary \ref{cormon} we obtain a monotone non increasing sequence
 $$   \Xi_0(\vp_0)\geq \ldots \geq \Xi_{l}(\vp_{l}) \geq \Xi_{l+1}(\vp_{l+1})  \ldots \geq c(\mu, \nu) \ . $$
 The iterations stop when this  sequence saturate, according to a pre-determined criterion.
  \subsection{Application for quadratic cost}\label{apqc}
As a demonstration, let us consider the special (but interesting) case of quadratic cost function $c(x,y)=|x-y|^2$ on Euclidean space $X=Y$.
We observe the trivial inequality   $|x-y|^2=2\min_{z\in X}[|x-z|^2+|y-z|^2]$. Hence we may approximate $|x-y|^2$ by
\be\label{c0}c^{Z_m}(x,y):= 2\min_{z\in Z_m}[|x-z|^2+|y-z|^2] \geq |x-y|^2\ \  . \ee
 So, we use  $c^{(1)}(x, z):= 2|x-z_z|^2, c^{(2)}( z_z,y):= 2|y-z_z|^2$.
 \par
 The updating (\ref{newc}) takes now a simpler form due to Remark \ref{remmean}. Indeed, $z^{l+1}_j$ is nothing but the center of mass
 $$ z_j^{l+1}=\frac{\sum_{i\in A_j^l} s_zx_z + \sum_{i\in B_j^l} \tau_zy_z}{\sum_{i\in A_j^l} s_z + \sum_{i\in B_j^l} \tau_z}\ .  $$
 \section{Some experiments with quadratic cost on the plane}\label{yifat}
 \begin{figure}[p]
    \centering
    \includegraphics[width=7cm]{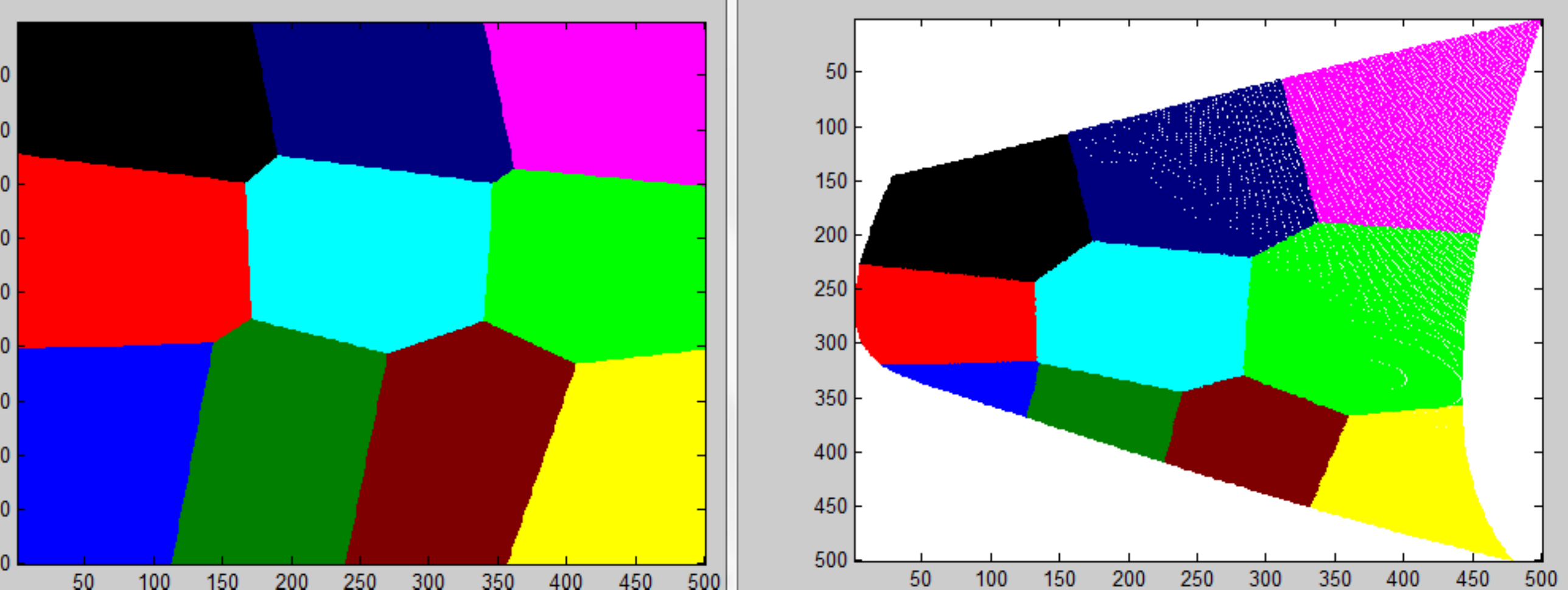}
       \caption{Partition for $\Phi$, $\lambda=0.2$ }
       \label{1p6}
\end{figure}

\begin{figure}[p]
   \centering
    \includegraphics[width=7cm]{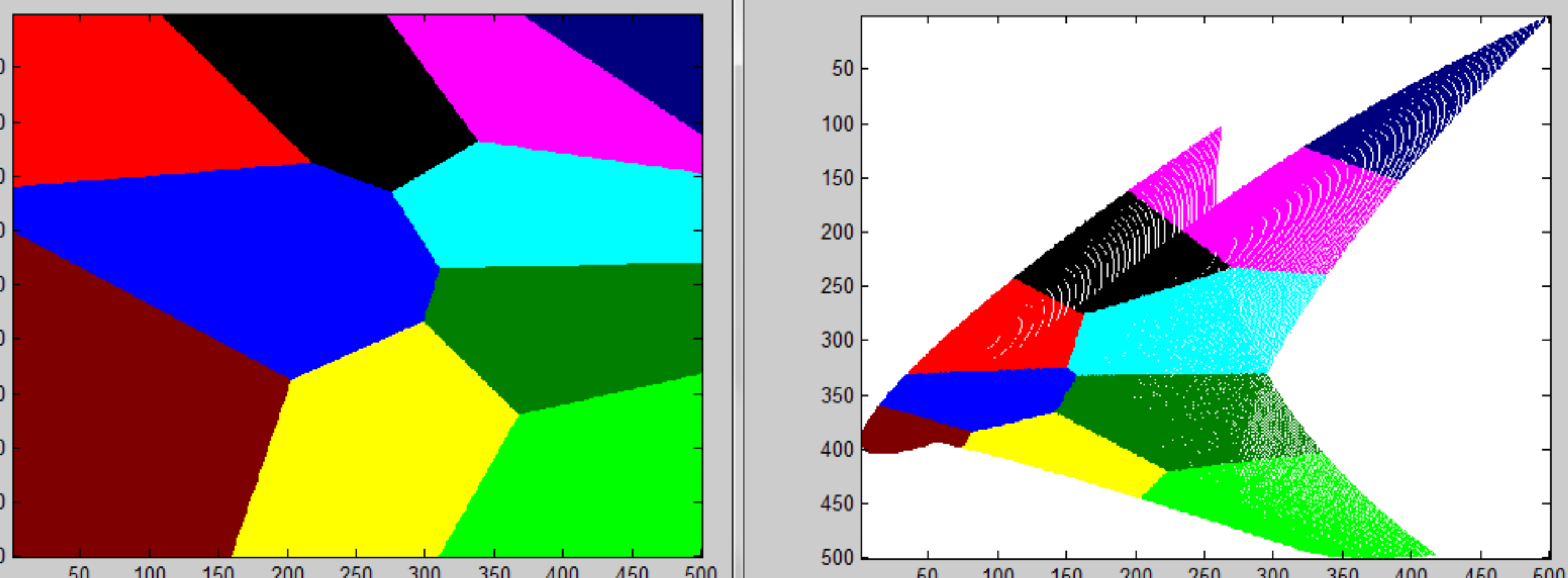}
       \caption{Partition for $\Phi$, $\lambda=1.2$}
       \label{4p7}
      \end{figure}

 \begin{figure}[p]
    \centering
    \includegraphics[width=7cm]{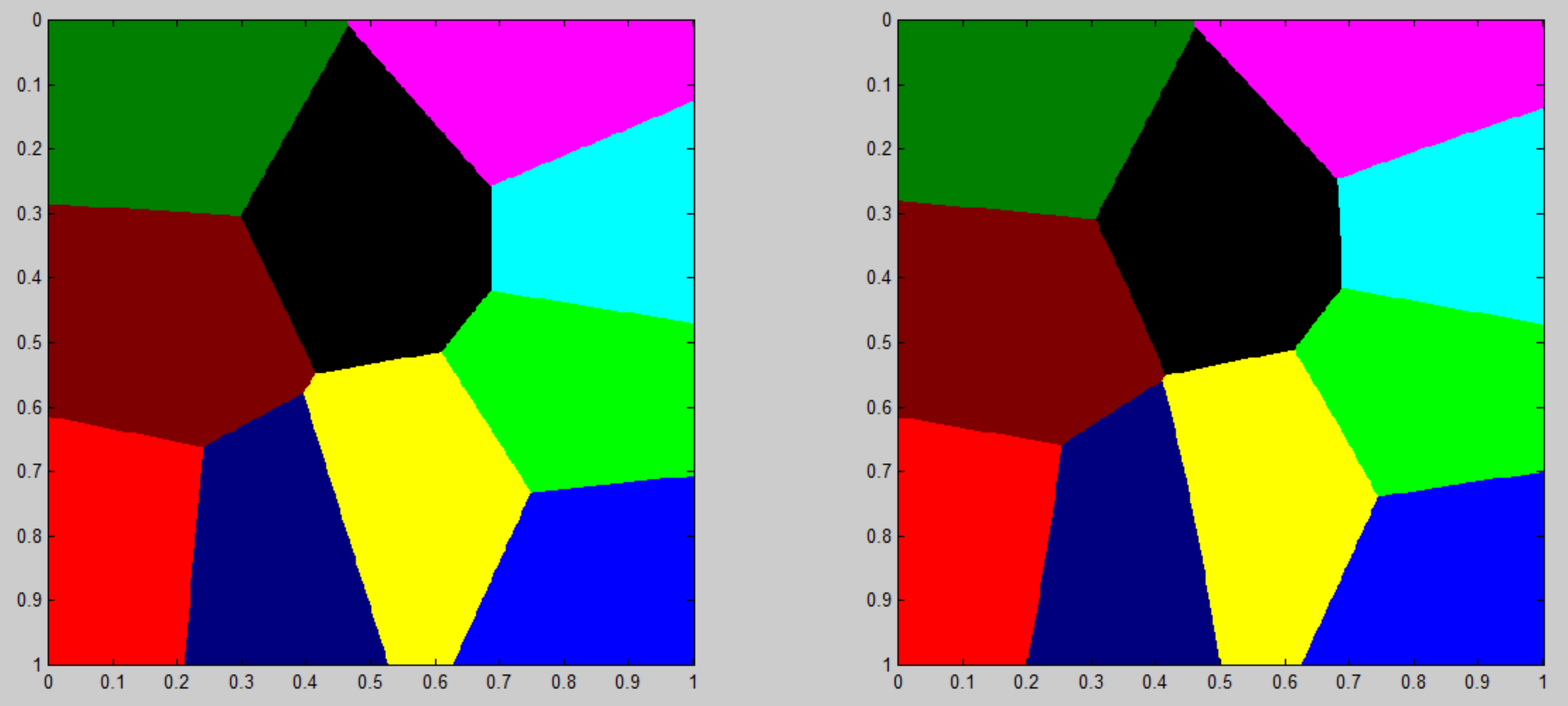}
       \caption{Comparison of partitions $\lambda=0.05$}
       \label{47}
\end{figure}
\begin{figure}[p]
    \centering
    \includegraphics[width=7cm]{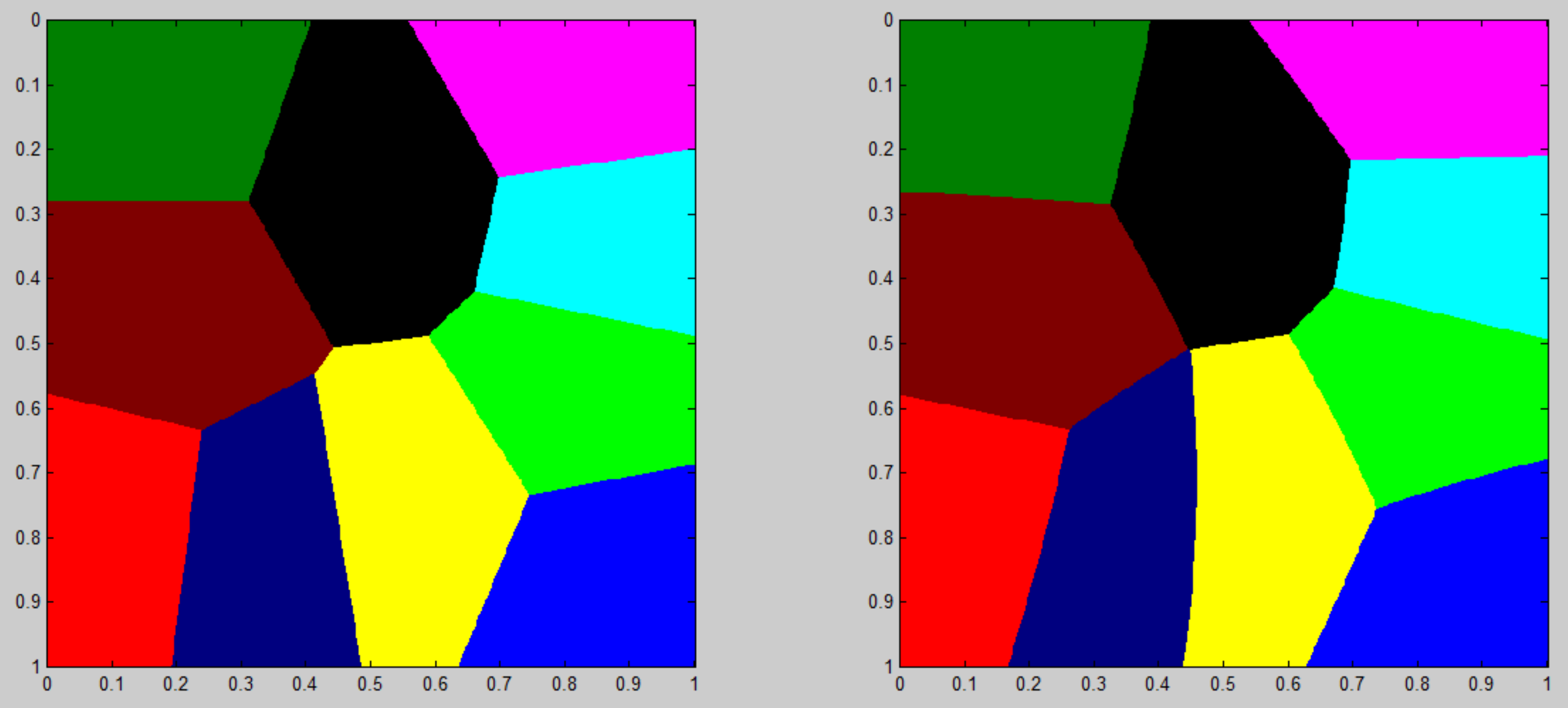}
       \caption{Comparison of partitions $\lambda=0.1$}
       \label{48}
\end{figure}

\begin{figure}[p]
    \centering
    \includegraphics[width=7cm]{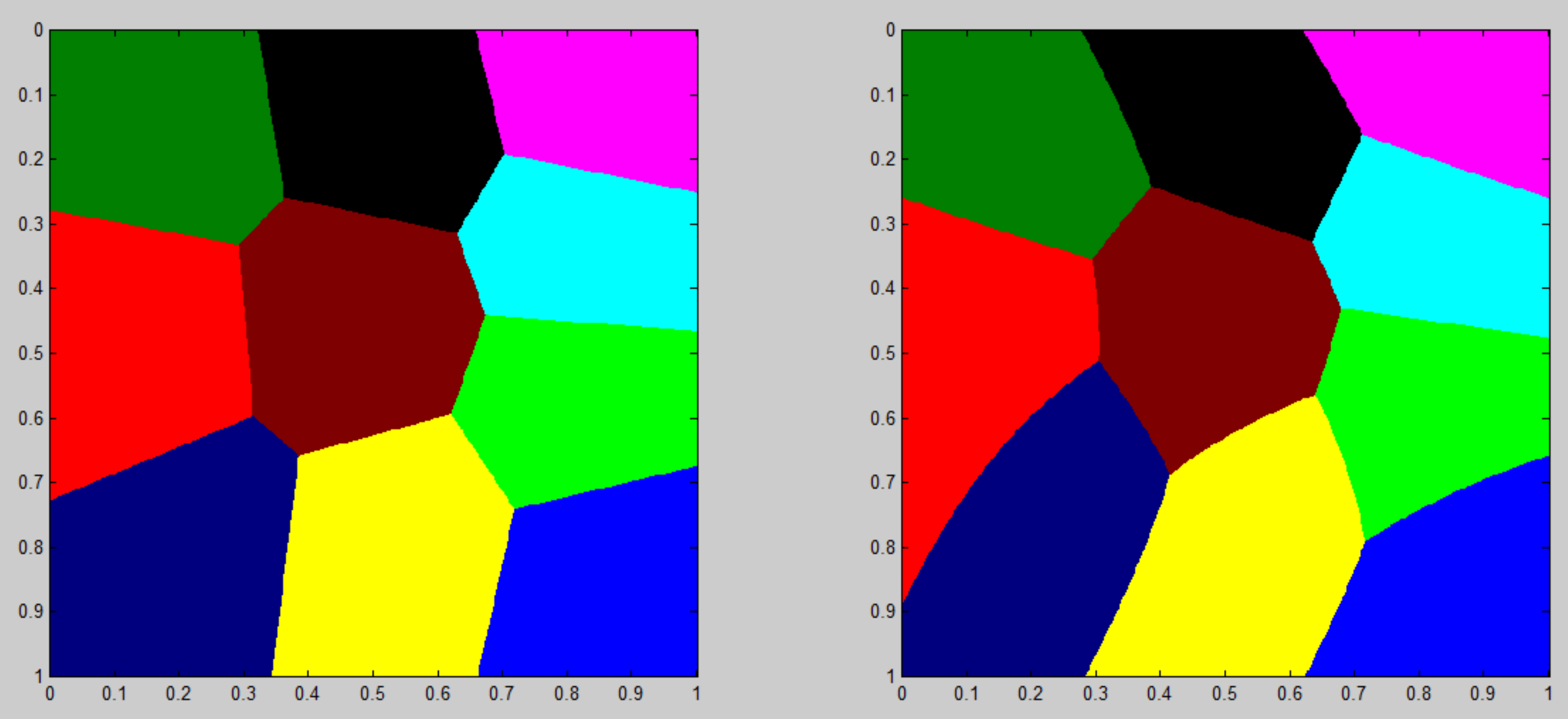}
       \caption{Comparison of partitions $\lambda=0.2$}
       \label{49}
\end{figure}

\begin{figure}[p]
    \centering
    \includegraphics[width=7cm]{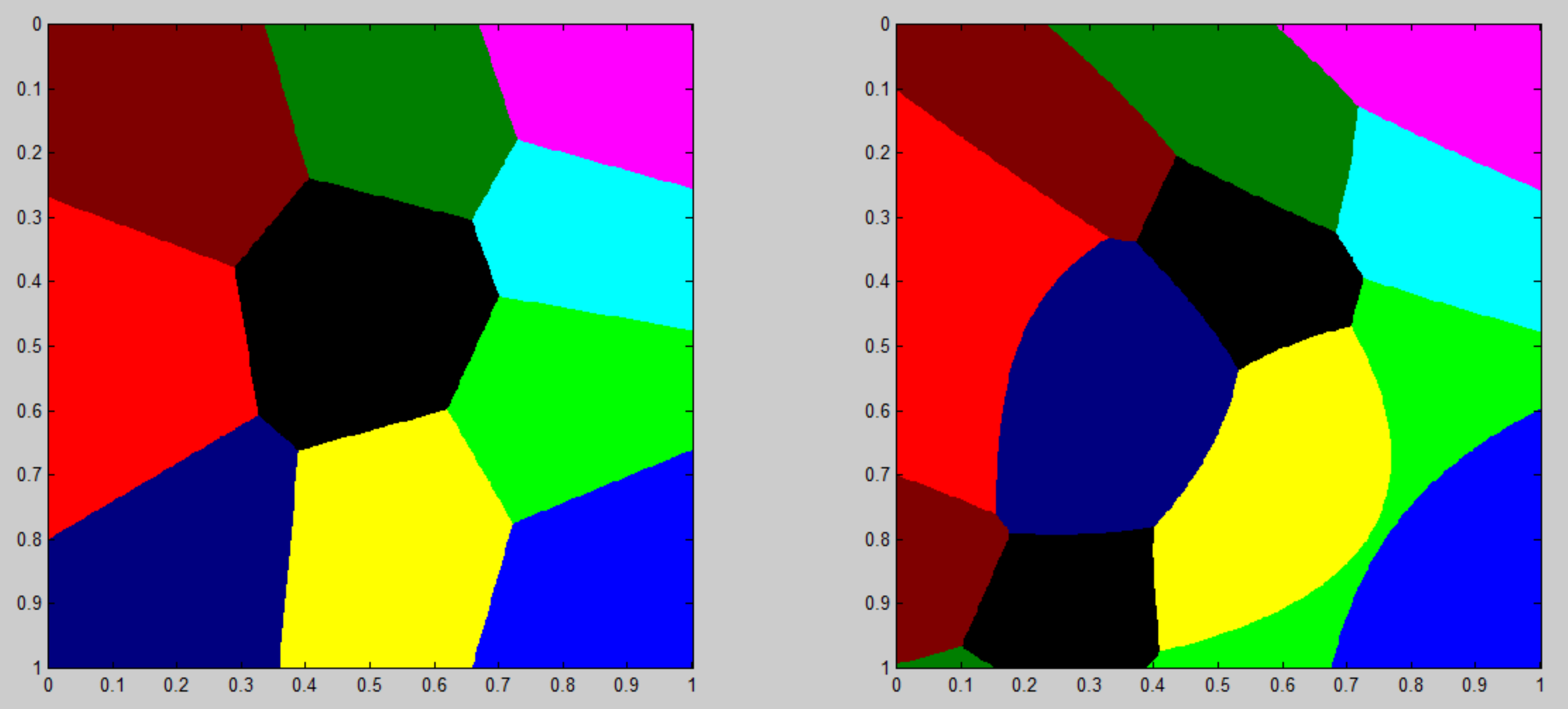}
       \caption{Comparison of partitions $\lambda=0.5$}
       \label{50}
\end{figure}

%
%
%
%
%
%
%
%
%
%
%
%

 In this section we demonstrate the algorithm for quadratic cost. The pair $(\mu,X)$ is always considered to be uniform Lebesgue measure  on the unit square $B:=\{(x_1, X_2); \ 0\leq x_1, x_2)\leq 1\}$. It is sampled by an empiric measure of regular grid composed on 400 points $x^{(i)}_1=i/20, x_2^{(j)}=j/20$,  $\mu(\{i/20, j/20)\})=1/400$, \ $1\leq i,j\leq 20$. The image space  $(Y,\nu)$ is, again, a probability measure on the plane  which depends on the particular experiment. The number of centers $m=10$ and their initial choice is arbitrary within the unit square.
 \par
 In the first experiments we used a given mapping $T:= (T_1, T_2):B\rightarrow \R^2$, and defined $(Y,\nu)$ according to $Y=T(B)$, $\nu=T_\#\mu$. In that case the naturel sampling is just $(y^{(i)}_1, y^{(j)}_2)=(T_1(i/20), T_2(j/20))$, and $\nu(\{(y^{(1)}_z, y^{(2)}_j)\})=1/400$.

 In all these experiment we used $T_k=\partial\Phi/\partial x_k$, $k=1,2$, where
   $\Phi(x_1, x_2) = 0.5 (x_1^2+x_2^2) + \lambda(\cos(x_1+2x_2)-\sin(x_1-x_2))$. Figs.\ref{1p6}-\ref{4p7} shows the saturated result for different values of $\lambda$.
   \par
   Fig \ref{47}-\ref{52} show pair of partitions on the $X$ square. The right  square is the image under $(\nabla\Phi)^{-1}$ of the partition in the left square. Note that for small values of $\lambda$ the two partitions  looks identical. This is, in fact, what we expect as long as $\Phi$ is a convex function. Indeed, the celebrated Brenier's theorem of polar factorization [\ref{[Br]}] implies just this! For larger values of $\lambda$, $\Phi$ is not convex and we see clearly the difference between these two partitions.
 \par
 In the second class of experiments we used different   domains for $Y$ (e.g. $T$ shaped, $I$ shaped and $A$ shaped) which are not induced by a mapping.   Fig. \ref{20p12}-\ref{15p10} display the induced partitions after saturation for different initial choices of the centers $z_z$. It demonstrates that the saturated partition may depend on the initial choice of the centers.

\section{Comparison with other semi discrete algorithms}\label{compare}
 Applications of semi-discrete methods  for numerical algorithms  where introduced in paper by   M\'{e}rigot [\ref{Me}], followed by a paper of L\'{e}vy [\ref{Levi}]. Here we indicate the similar and different aspects of our proposed algorithm, compared to [\ref{Me}, \ref{Levi}].
 \par
 The starting point of  M\'{e}rigot-L\'{e}vy algorithm for quadratic cost involves a discretization  $\nu_m$ of the target measure $\nu$. For $\nu_m=\sum_1^m r_i\delta_{y_i}$, the optimal plan for transporting $\mu$ is obtained by maximizing
 \be\label{ML} \R^m\ni\vp \mapsto \int\min_{1\leq i\leq m} [|x- y_i|^2+p_i] \mu(dx) - \sum_i^mr_ip_i \ . \ee
 This is equivalent to the function we defined (for the special case of quadratic cost) as $\Xi_\mu^{Z_m}(\vp)-\vp\cdot\vec{r}$, whose maximum over $\R^m$ is
$(-\Xi_\mu^{Z_m})^*(-\vec{r})$ as defined in (\ref{Xi*mu}). The optimal partition induced by maximizing (\ref{ML}) is refined by taking finer and finer discretization of $\nu$ with increasing number of points $m$. The multi-grid method is, essentially, using the data of the maximizer $\vp$ corresponding to $\nu_m$ as an initialization for the   $m+1$ level maximization corresponding to (\ref{ML}).
\par
In the present paper we take a different approcah, namely the  semi-discretization of the {\it cost function} $c=c(x,y)$ via  (\ref{sd}). It is, in fact, equivalent to a two sided discretization analogus to (\ref{ML}) (in the quadratic case), as we can observe from (\ref{c=Xi}). However, by carrying the duality method one step forward we could  reduce the optimization problem to a single one over $\R^m$ via Theorem \ref{firstdu}.

\newpage

\begin{center}{\bf References}\end{center}
\begin{enumerate}
\item\label{[A1]} L. Ambrosio: {\it
Lecture notes on optimal transport problems
Mathematical Aspects of Evolving Interfaces},  Lecture Notes in Math., Funchal, 2000, vol. 1812, Springer-Verlag, Berlin (2003), pp. 1-52
\item\label{[BC]}   Bauschke, H.H and  Combettes, P.L: {\it  Convex Analysis and Monotone Operator Theory in Hilbert Spaces}, Springer, (2011).
\item \label{[Br]} Y. Brenier: {\it Polar factorization and monotone
rearrangement of vector valued  functions}, Arch. Rational Mech
\item\label{Levi} Bruno L\'{e}vy: {\it  A numerical algorithm  for $L_2$ semi-discrete optimal transport in  3D}, arXiv:1409.1279v1  [math.AP]  3 Sep 2014
\item\label{[B]} Bogachev, V.I {\it Measure Theory}, Springer 2007
\&Anal., {\bf 122}, (1993), 323-351.
 \item\label{Cap1} Caffarelli, L; Gonz\'{a}lez, M and   Nguyen, T.: {\it A perturbation argument for a Monge-Amp\'{e}re type equation arising in optimal transportation},  Arch. Ration. Mech. Anal. 212 (2014), no. 2, 359-414
\item\label{[Hed]}
Chiappori, Pierre-Andr\'{e}, Robert J. McCann, and Lars P. Nesheim. "Hedonic price equilibria, stable matching, and optimal transport: equivalence, topology, and uniqueness." Economic Theory 42.2 (2010): 317-354
\item\label{[AF]}  Frank, A:
{On Kuh\'{n}s Hungarian method-a tribute from Hungary},
Naval Research Logistics, 52 (1) (2005), 2-6
\item\label{GO} Glimm, T., and V. Oliker. {\it Optical design of single reflector systems and the Monge-Kantorovich mass transfer problem},  Journal of Mathematical Sciences 117.3 (2003): 4096-4108
\item\label{[GL]} Graf, S and Luschgy, H.: {\it Foundations of Quantization for Probability Distributions}, Lect. Note Math. 1730, Springer, (2000)
\item\label{[HWK]} Kuhn, H.W.:
{\it The Hungarian method for the assignment problem},
Naval Research Logistics Quarterly, 2 (1,2) (1955), 83-97
 \item\label{[HWK1]}  Kuhn, H.W:
{\it Statement for Naval Research Logistics},
Naval Research Logistics, 52 (1) (2005), p. 6
 \item\label{[Kan]}    Kantorovich, L:  {\it On the translocation of
masses}, C.R (Doclady) Acad. Sci. URSS (N.S), 37, (1942), 199-201
\item \label{Me}  M\'{e}rigot, Q.: {\it  A multiscale approach to optimal transport},  Computer Graphics Forum, Wiley-Blackwell, 2011, 30 (5), pp.1584-1592.
\item\label{mon} Monge,G:\ {\it M\'{e}moire sur la th\'{e}orie des d\'{e}blais et des remblais}, In  Histoire de l\'{A}cad\'{e}mie Royale des Sciences de Paris,  666-704, 1781
\item  \label{[JM]}  Munkres, J:: {\it Algorithms for the assignment and transportation problems}, Jo Soc. IDUST. APPL. HATtt.
Vol. 5, No. 1, March, 1957
\item\label{PCK} Papadimitriou, C. H and Kenneth S:. {\it Combinatorial optimization: algorithms and complexity}. Courier Dover Publications, 1998
\item \label{[PD]} Pentico, D. W: {\it Assignment problems: a golden anniversary survey}, European J. Oper. Res. 176 (2007), no. 2, 774-793.

\item\label{Ru1}  Rubinstein, J.,  Wolansky, G.: {\it Intensity control with a free-form lens},  J. Opt. Soc. Amer. A 24 (2007), no. 2, 463-469
 \item\label{Ru2} Rubinstein, J., Wolansky, G.: {\it A weighted least action principle for dispersive waves},  Ann. Physics 316 (2005), no. 2, 271-284
\item\label{[RYT]} Rubner, Y.,  Tomasi, C. and Guibas, L.: {\it The Earth Mover's Distance as a Metric for Image Retrieval}, International Journal of Computer Vision, {\bf 40}, 2,  99-121 (2000),

\item \label{[V1]}  Villani, C:  {\it Topics in Optimal Transportation}, A.M.S Vol 58, 2003
\item\label{[VO]}  Votaw, D.F and  Orden, A: {\it The personnel assignment problem}, Symposium on Linear Inequalities and Programmng, SCOOP 10, US Air Force, 1952, 155-163
\item \label{[W]} Wolansky, G.: {\it On Semi-discrete Monge Kantorovich and Generalized Partitions}, to appear in JOTA



    \item\label{[Z]} P.L. Zador, {\it Asymptotic quantization error of continuous signals and the quantization dimension},  IEEE Trans. Inform. Theory
28, Special issue on quantization, A. Gersho \& R.M. Grey Eds. (1982) 139-149.
\end{enumerate}
\end{document}